\documentclass[11pt,reqno]{amsart}
\usepackage{amsfonts, amsbsy, amsmath, amssymb, mathtools, amsthm,ragged2e}
\usepackage{amscd}
\usepackage{color,enumerate}
\usepackage{breqn}
\usepackage{pgfplots}
\pgfplotsset{compat=1.15}
\usepackage{tikz}
\usepackage{tikz-cd}
\usetikzlibrary{arrows}
\usepackage{verbatim}
\usepackage{graphicx}
\usetikzlibrary{matrix,calc}
\usepackage{caption}
\usepackage{hyperref}
\usepackage{float}
\usepackage[margin=2.5cm]{geometry}
\usepackage{epsfig}

\newtheorem{thm}{Theorem}[section]
\newtheorem{lem}[thm]{Lemma}

\newtheorem{thm-con}[thm]{Theorem-Conjecture}
\numberwithin{equation}{section}

\theoremstyle{definition}
\newtheorem{defn}[thm]{Definition}
\newtheorem{rmk}[thm]{Remark}
\newtheorem{notation}[thm]{ Notation}

\newtheorem{exmp}[thm]{Example}

\newtheorem{set}[thm]{Set-up}

\DeclareMathOperator{\del}{del}
\DeclareMathOperator{\link}{link}
\DeclareMathOperator{\Tor}{Tor}
\DeclareMathOperator{\reg}{reg}

\newcommand{\bk}{\ensuremath{{\bf k}}}

\begin{document}
	
\title[Powers of vertex cover ideals of Simplicial Trees]{Powers of vertex cover ideals of Simplicial Trees}
\author{Bijender}
\email{2019rma0011@iitjammu.ac.in}
\author[Ajay Kumar]{Ajay Kumar}
\email{ajay.kumar@iitjammu.ac.in}
\author[Rajiv Kumar]{Rajiv Kumar}
\email{rajiv.kumar@iitjammu.ac.in}
\address{Department of Mathematics, Indian Institute of Technology Jammu, J\&K, India - 181221.}
\date{\today}



\maketitle
\begin{abstract}
   In $2011$, Herzog, Hibi, and Ohsugi conjectured that if $J$ is the cover ideal of a chordal graph, then $J^s$ is componentwise linear for all $s \ge 1.$ In 2022, H\`a  and Tuyl considered objects more general than chordal graphs and posed the following  problem:  Let $J(\Delta)$ be the cover ideal of a simplicial tree $\Delta.$ Is it true that  $J(\Delta)^s$ is componentwise linear for
   all $s \geq 1?$ In this article, we give an affirmative answer to this problem.
\end{abstract}
\section{Introduction}
The study of homological and combinatorial properties of monomial ideals is a  beautiful area with rich interaction between combinatorics and commutative algebra. In particular,  the class of  squarefree monomial ideals is strongly connected to combinatorics. A basic example of a squarefree monomial ideal is the edge ideal of a simple hypergraph $\mathcal{H}.$  Recall that a simple hypergraph $\mathcal{H}$  is a pair $(\mathcal{V},\mathcal{E})$, where $\mathcal{V}$ be a finite set and $\mathcal{E}$ be a collection of non-empty subsets of $\mathcal{V}$ such that for any $F,G\in \mathcal{E}$ with $F \neq G$, we have $F\not\subset G.$ The \emph{edge ideal} of $\mathcal{H}$, denoted by $I(\mathcal{H})$, is the squarefree monomial ideal in the polynomial ring $S = \mathbb{K}[x_1,\dots,x_n]$ over a field $\mathbb{K}$ given by $
I(\mathcal{H}) = \langle  \prod\limits_{x \in F}x: F\in \mathcal{E} \rangle.$ The Alexander dual of $I(\mathcal{H})$, is called the\textit{ vertex cover ideal} of $\mathcal{H}$ denoted as $I(\mathcal{H})^{\vee}=J(\mathcal{H}).$ The assignment $\mathcal{H} \leftrightarrow I(\mathcal{H})$ gives a one-to-one correspondence between the family of hypergraphs and the family
of square-free monomial ideals. Faridi \cite{Faridi05} introduced the notion of edge ideal $I(\mathcal{H})$ of a hypergraph $\mathcal{H}$ in terms of the facet ideal of a simplicial complex.  Let $\Delta$ be a simplicial complex on the vertex set $\mathcal{V}=\{x_1,x_2,\ldots,x_n\},$ then one can associate a square-free monomial ideal 
$I(\Delta)=\langle \prod\limits_{x \in F}x: F \in  \mathcal{F}(\Delta)\rangle$, called the facet ideal of $\Delta$, where $\mathcal{F}(\Delta)$ denotes the set of all facets of $\Delta.$ This correspondence is very effective in the study of powers of squarefree monomial ideals. To each simplicial complex $\Delta$, we can associate a hypergarph, namely, $\mathcal{H}(\Delta) = (\mathcal{V},\mathcal{F}(\Delta)).$ Observe that $I(\mathcal{H}(\Delta)) =
I(\Delta).$ The \emph{vertex cover ideal} of $\Delta$ is defined as $J(\Delta) = J(\mathcal{H}(\Delta))=I(\mathcal{H}(\Delta))^{\vee}.$ Using this approach one can view any squarefree monomial ideal as the facet set of a simplicial complex or the edge ideal of a hypergraph. Another correspondence between squarefree monomial ideals and simplicial complexes is called the Stanley-Reisner
correspondence. Let $\Delta$ be a simplicial complex on the vertex set $V=\{x_1,\dots,x_n\}.$ Then the \emph{Stanley-Reisner ideal} of $\Delta$, denoted as $I_{\Delta}$, is defined as
$
I_{\Delta} = \langle x_{i_1}\cdots x_{i_s}:\{x_{i_1},\dots,x_{i_s}\}\notin \Delta\rangle.
$   A simplicial complex is said to be (sequentially) Cohen-Macaulay if the quotient ring $S/I_{\Delta}$ has the same property.  Eagon and 
Reiner \cite{eagon} proved that a simplicial complex $\Delta$ is Cohen-Macaulay if and only if $I_{\Delta}^{\vee}$ has a linear resolution. Ideals having a linear resolution are considered to be computationally simple.   Herzog and Hibi \cite{HHibi}  introduced a new class of ideals called componentwise linear ideals. A graded ideal $I$ is said to be componentwise linear if the ideal $I_{<j>}$ generated by the homogeneous elements of degree $j$ in $I$ has a linear resolution for all $j \geq 0.$ Herzog-Hibi \cite{HHibi} generalized  the result of Eagon and Reiner by showing that a simplicial complex $\Delta$ is sequentially Cohen-Macaulay if and only if the ideal $I_{\Delta}^{\vee}$ is componentwise linear. It is a known fact that over the field of characteristic $0$, a graded $I$ in a polynomial ring $S$ is Koszul if and only if it is componentwise linear. Koszul modules were introduced Herzog and Iyengar in \cite{HIS}. Koszul modules have significant applications in areas of commutative algebra and algebraic geometry.  

The problem of finding classes of componentwise linear ideals  is of great interest and has been studied extensively by many authors in the literature (see \cite{Er2,EQ2019,HHO2011,Mohammadi,FM,MMS}). In particular, one would like to find some combinatorial conditions on a graph $G$ such that powers of its cover ideal are componentwise linear. Francisco and Tuyl  \cite{fv} proved that the vertex cover ideal of a chordal graph is componentwise linear. Herzog, Hibi and Ohsugi  \cite{HHO2011} studied the powers of the cover ideals associated to graphs and shown that all powers of the vertex cover ideal of a Cohen-Macaulay chordal graph are componentwise linear. Also, they  proposed a conjecture that all powers of the vertex cover ideal of a chordal graph are componentwise linear. There has been a very little progress made on this conjecture except for some classes of graphs like  generalized star graphs, Cohen-Macaulay chordal graphs, trees, Cohen-Macaulay bipartite and cactus  graphs, $(C_4,2K_2)$-free graphs (see \cite{Er2, HHO2011,KumarAR,Mohammadi,FM,MMS}). H\`a  and Tuyl  \cite{taivan} posed the following question: Let $J(\Delta)$ be the cover ideal of a simplicial tree $\Delta.$ Is it true that  $J(\Delta)^s$ is componentwise linear for
all $s \geq 1?$ In this article, we give an affirmative answer to this question, and hence establish the generalization of Herzog, Hibi and Ohsugi's conjecture for the special case of simplicial trees.

\begin{thm}\label{Th-1stMainResult}
	Let $\Delta$ be a simplicial tree. Then for every $k \in \mathbb{N}_{>0}$, $J(\Delta)^k$ has linear quotients, and hence it is componentwise linear.
\end{thm}

Another problem of interest is to investigate properties of an ideal $I$ that are preserved under taking powers. It is
natural to ask that if a graded ideal I is componentwise linear, then does the ideal $I^s$ also componentwise linear for all $s \geq 1$? Herzog, Hibi and Zheng  \cite{HHZ} prove that if $I(G)$ has a linear resolution, then so does every power of $I(G).$ However, there are examples such that $I$ has a linear resolution (so is componentwise
linear), but $I^2$ does not have this property. First such example is due to Terai. Terai observed that, if the characteristic of the base field is zero, then the Stanley Reisner ideal $I=(x_1x_2x_3,x_1x_2x_5, x_1x_3x_6, x_1x_4x_5, x_1x_4x_6, x_2x_3x_4, x_2x_4x_6, x_2x_5x_6, x_3x_4x_5, x_3x_5x_6)$ has a linear resolution, while $I^2$ has no linear resolution. However, if the characteristic of the base field is $2,$ then $I$ does not have a linear
resolution.  In \cite{Sturmfels} Sturmfels gave an example of a square-free monomial ideal $I= (x_4x_5x_6, x_3x_5x_6, x_3x_4x_6, x_3x_4x_5, x_2x_5x_6, x_2x_3x_4, x_1x_3x_6, x_1x_4x_5)$ with  linear resolution (independent of the characteristic of the base field). But $I^2$ does  not have a linear resolution. As a consequence of Theorem \ref{Th-1stMainResult} and  \cite[Corollary 7.8]{Faridi05}, we obtain the second main result of this article.
\begin{thm}\label{Th-2ndMainResult}
Let $\Delta$ be a simplicial tree. Then the following are equivalent. 
\begin{enumerate}[(i)]
	\item $J(\Delta)$ has a linear resolution.
	\item $J(\Delta)^s$ has a linear resolution for some $s \geq 1.$
	\item $J(\Delta)^s$ has a linear resolution for all $s \geq 1.$
	\item $S/I(\Delta)$ is Cohen-Macaulay.
	\item $\Delta$ is unmixed.
	\item $\Delta$ is grafted.
\end{enumerate}
\end{thm}

\section{Preliminaries}
For simplicity, we denote the set $\{1,\dots,n\}$  by $[n].$ Further, we use notation $[n]^a$ for the set $\{(k_1,\dots,k_a): k_i\in [n]\}.$   

Let $\Delta$ be a simplicial complex. The set of all facets of $\Delta$ is denoted by $\mathcal{F}(\Delta).$ If $\mathcal{F}(\Delta)=\{F_1,\dots,F_t\}$, then we write $\Delta=\langle F_1,\dots,F_t\rangle.$ A simplicial complex with only one facet is called a \emph{simplex}. 
By a \emph{subcollection} of $\Delta$, we mean a simplicial complex whose facet set is a subset of the facet set of $\Delta.$
	
\begin{defn}
    Let $\Delta$ be a simplicial complex on the vertex set $\mathcal{V}$ with $\mathcal{F}(\Delta) =\{F_1,\dots,F_t\}.$
\begin{enumerate}[(a)]
	\item For a face $F$ of $\Delta$, the \emph{deletion} of $F$, denoted as $\del_{\Delta}(F)$, is a simplicial complex defined as
	\[
	\del_{\Delta}(F) = \{ G \in \Delta : G \cap F = \emptyset\}. 
	\]
			
	\item For a face $F$ of $\Delta$, the \emph{link} of $F$, denoted as $\link_{\Delta}(F)$, is a simplicial complex defined as 
	\[
	\link_{\Delta}(F) = \{G\in \Delta : G\cap F = \emptyset, G\cup F \in \Delta\}.
	\]
	
	\item A vertex $x \in \mathcal{V}$ is called a \emph{shedding vertex} of $\Delta$ if there is no face in $\link_{\Delta}(x)$ which is a facet in $\del_{\Delta}(x).$
	$\Delta$ is said to be \emph{vertex decomposable} if it is either a simplex or it has shedding vertex $x$ such that $\del_{\Delta}(x)$ and $\link_{\Delta}(x)$ are vertex decomposable.
	
	\item $\Delta$ is said to be \emph{shellable} if exists a linear order $F_{i_1}\dots,F_{i_t}$ of all facets of $\Delta$ such that for all $r,s\in [t]$ with $ r<s$, there exists $x \in F_{i_s}\setminus F_{i_r}$ and $j\in [s-1]$ with $F_{i_s}\setminus F_{i_j} =\{x\}.$
			
	\item $\Delta$ is said to be \emph{connected} if for every $i\neq j$, there exists a sequence of facets $F_{t_1},\dots,F_{t_r}$ of $\Delta$ such that $F_{t_1} = F_i$ ,  $F_{t_r} = F_j$ and $F_{t_s} \cap F_{t_{s+1}} \neq \emptyset$   for $s \in [r-1].$
			
	\item A facet $F$ of $\Delta$ is called a \emph{leaf} if either $F$ is the only facet of $\Delta$, or there exists facet $G \in \mathcal{F}(\Delta)$ with $G \neq F$ such that $F\cap H \subset F\cap G$ for every $H \in \Delta$ with $H \neq F.$ Further, $G$ is called a \emph{branch} of F. In particular, if $F\cap G\ne\emptyset$, then $G$ is called a \emph{joint} of $F.$
			
	\item A leaf $F$ of $\Delta$ is called a \emph{good leaf} if the set $\{ F\cap G : G \in \mathcal{F}(\Delta) \}$ is totally ordered with respect to inclusion.
			
	\item $\Delta$ is called a \emph{simplicial tree} if it is connected and every non-empty subcollection of $\Delta$ has a leaf. If $\Delta$ is not connected, and every component of $\Delta$ is a tree, then $\Delta$ is called a \emph{forest}.
	
	\item An order $F_1,\dots,F_t$ of all facets of $\Delta$ is called a \emph{good leaf order} if for each $i \in [t-1]$, $F_i$ is a good leaf of subcollection $\langle F_i,\dots,F_t\rangle$ of $\Delta.$
	\end{enumerate}
\end{defn}

\begin{defn}
	A simplicial complex $\Delta$ is said to be \emph{grafting} of a simplicial complex 
	$\Delta'$ with simplices $F_1,\dots,F_s$ (or we say $\Delta$ is \emph{grafted}) if the following conditions are satisfied.
	\begin{enumerate}[(a)]
		\item $\Delta =\langle F_1,\dots,F_s\rangle \cup\Delta'$;
		\item $\mathcal{V}(\Delta')\subset \bigcup_{i=0}^{s}\mathcal{V}(F_i)$;
		\item $F_i$ is a leaf of $\Delta$ for all $i \in [s]$;
		\item $F_i \notin \mathcal{F}(\Delta')$ for all $i \in [s]$;
		\item $F_i\cap F_j =\emptyset$ for $i\neq j$;
		\item  If $G \in \mathcal{F}(\Delta')$ is a joint of $\Delta$, then $\Delta\setminus G$ is also grafted.
	\end{enumerate}
\end{defn}

\begin{defn}
    Let $\mathcal{V}$ be a finite set and $\mathcal{E}$ be a collection of non-empty subsets of $\mathcal{V}.$ Then the pair $\mathcal{H}=(\mathcal{V},\mathcal{E})$ is called a \emph{hypergraph}. Elements of $\mathcal{V}$ are called \emph{vertices} and elements of $\mathcal{E}$ are called \emph{edges}. We will use $\mathcal{V}(\mathcal{H})$ and $\mathcal{E}(\mathcal{H})$ to denote the sets of vertices and edges of $\mathcal{H}$, respectively. A hypergraph $\mathcal{H}$ is called \emph{simple} if for any $F,G\in \mathcal{E}(\mathcal{H})$ with $F \neq G$, we have $F\not\subset G.$ An edge $E$ is called \emph{trivial} if it contains only one vertex. A vertex $x \in \mathcal{V}(\mathcal{H})$ is called \emph{isolated} if $\{x\} \in \mathcal{E}(\mathcal{H})$, or $x \notin E$ for all $E \in \mathcal{E}(\mathcal{H}).$
    A hypergraph $\mathcal{H}$ is called \emph{isolated} if every vertex of $\mathcal{H}$ is isolated. Simple hypergraphs are also called \emph{clutters}.
\end{defn}

All hypergraphs considered in this paper are simple.

\begin{notation}
	Let $\mathcal{H}$ be a hypergraph. Then the hypergraph obtained by deleting its isolated vertices is denoted by $\mathcal{H}^\circ.$
\end{notation}

\begin{defn}
	Let $\mathcal{H}=(\mathcal{V}(\mathcal{H}),\mathcal{E}(\mathcal{H}))$ be a hypergraph. 
	\begin{enumerate}[(a)]
		\item A subset $C$ of $\mathcal{V}(\mathcal{H})$ is called a \emph{vertex cover} of $\mathcal{H}$, if $C \cap E \neq \emptyset$ for all $E \in \mathcal{E}(\mathcal{H}).$ A vertex cover $C$ of $\mathcal{H}$ is called \emph{minimal} if there is no proper subset of  $C$ which is a vertex cover of $\mathcal{H}.$
		
		\item A subset $W$ of $\mathcal{V}(\mathcal{H})$ is called an \emph{independent set} if $E\cap (\mathcal{V}(\mathcal{H}) \setminus W)\neq \emptyset$ for all 
		$E\in \mathcal{E}(\mathcal{H}).$ 
	\end{enumerate}
\end{defn}

\begin{defn}
    Let $\Delta$ be a simplicial complex and $\mathcal{H}$ be a hypergraph.
	    \begin{enumerate}[(a)]
		    \item The \emph{independence complex} of $\mathcal{H}$, denoted by $\Delta(\mathcal{H})$, is the simplicial complex whose faces are independent sets of $\mathcal{H}.$

            \item The hypergraph \emph{associated} to $\Delta$, denoted by $\mathcal{H}(\Delta)$, is the hypergraph whose edge set is $\mathcal{F}(\Delta).$
            
            \item If cardinality of each minimal vertex cover of $\mathcal{H}$ is same, then we say that $\mathcal{H}$ is \emph{unmixed}. Further,  $\Delta$ is called unmixed if $\mathcal{H}(\Delta)$ is unmixed.
	    \end{enumerate}
\end{defn}

Let $S = \mathbb{K}[x_1,\dots,x_n]$ be a polynomial ring over a field $\mathbb{K}.$

\begin{defn}
    Let $\mathcal{V} =\{x_1,\dots,x_n\}$ be a finite set. For a subset $F\subset \mathcal{V}$, let $x_F$ denotes the monomial $\prod_{x_j\in F} x_j$ in $S$, and $\wp_F$ denotes the prime ideal of $S$ generated by the variables $x_j \in F.$ 
\begin{enumerate}[(a)]
	\item Let $\mathcal{H}$ be a hypergraph on the vertex set $\mathcal{V}.$ Then the \emph{edge ideal} of $\mathcal{H}$, denoted by $I(\mathcal{H})$, is the squarefree monomial ideal in $S$ given by
	$$
	I(\mathcal{H})=\langle x_F: F\in \mathcal{E}(\mathcal{H})\rangle.
	$$ 
    The Alexander dual $J(\mathcal{H})=I(\mathcal{H})^{\lor}$ of $I(\mathcal{H})$ is called the \emph{vertex cover ideal} of $\mathcal{H}.$ Note that $J(\mathcal{H})= \bigcap_{F \in \mathcal{E}(\mathcal{H})}\wp_F.$
    
	\item Let $\Delta$ be a simplicial complex on the vertex set $\mathcal{V}.$ Then the \emph{facet ideal} of $\Delta$, denoted by $I(\Delta)$, is the squarefree monomial ideal in $S$ given by
	$$
	I(\Delta) = \langle x_F: F\in \mathcal{F}(\Delta) \rangle.
	$$
	The  \emph{vertex cover ideal} of $\Delta$ is given by $J(\Delta) = J(\mathcal{H}(\Delta)).$ Note that $I(\Delta) = I(\mathcal{H}(\Delta)).$ Thus,
	$J(\Delta) = I(\Delta)^{\vee}.$
\end{enumerate}
\end{defn}

\begin{defn}
	Let $\mathcal{H} = (\mathcal{V}(\mathcal{H}),\mathcal{E}(\mathcal{H}))$ be a hypergraph. For $x_j\in \mathcal{V}(\mathcal{H})$, 
	\begin{enumerate}[(a)]
		\item a \emph{contraction} of $\mathcal{H}$, denoted by $\mathcal{H}/x_j$, is
		the hypergraph whose edge ideal is the squarefree monomial ideal $I(\mathcal{H}):x_j$, and
		
		\item a \emph{deletion} of $\mathcal{H}$, denoted by $\mathcal{H} \setminus x_j$, is
		the hypergraph whose edge ideal is the squarefree monomial ideal $I(\mathcal{H})\cap k[x_1,\dots,\hat{x_j},\dots,x_n].$ 
	\end{enumerate}
\end{defn}

Let us consider the following sets
    $$\mathcal{E}' =\{F\setminus \{x_j\}: F\in \mathcal{E}(\mathcal{H}),x_j\in F\}
    $$
and
    $$\mathcal{E}'' =\{E \in \mathcal{E}(\mathcal{H}): x_j\not{\in} E,F\setminus \{x_j\}\not \subset E ~\mbox{for all} ~ F\in \mathcal{E}(\mathcal{H})\setminus\{E\}\}.$$
Then the vertex set and the edge set of contraction $H/x_j$ (resp. deletion $H\setminus x_j$) of $\mathcal{H}$ are given by 
    $$\mathcal{V}(\mathcal{H}/x_j) = \mathcal{V}(\mathcal{H})\setminus \{x_j\} ~(\text{resp. } \mathcal{V}(\mathcal{H}\setminus x_j) = \mathcal{V}(\mathcal{H})\setminus\{x_j\})$$
and 
    $$\mathcal{E}(\mathcal{H}/x_j) = \mathcal{E}' \cup \mathcal{E}''
    ~(\text{resp. }\mathcal{E}(\mathcal{H} \setminus x_j) =
    \{E \in \mathcal{E}(\mathcal{H}): x\not{\in} E\}).$$
Thus the definition of a vertex decomposable simplicial complex translates to a vertex
decomposable hypergraph as follows.

\begin{defn}
	A hypergraph $\mathcal{H} = (\mathcal{V}(\mathcal{H}),\mathcal{E}(\mathcal{H}))$ is said to be \emph{vertex decomposable} if either it is an isolated hypergraph or there exists a vertex $x\in \mathcal{V}(\mathcal{H})$ such that
	\begin{enumerate}[(a)]
		\item $\mathcal{H}/x$ and $\mathcal{H} \setminus x$ are vertex decomposable, and
		
		\item for every independent set $W$ in $\mathcal{H}/x$, there exists an independent set $W'$ in $H\setminus x$ such that $W\subsetneq W'.$
	\end{enumerate}
\end{defn}

\begin{rmk}\label{R-ContRemark}
	Let $\mathcal{H}$ be a simple hypergraph on the vertex set $\mathcal{V}(\mathcal{H})$ and the edge set $\mathcal{E}(\mathcal{H}).$ Then for every $x \in \mathcal{V}(\mathcal{H})$ and $E \in \mathcal{E}(\mathcal{H})$, there exists an edge $E' \in \mathcal{E}(\mathcal{H}/x)$ such that $E' \subset E \setminus \{x\}.$
\end{rmk}
    
\begin{defn}
	Let $M$ be a finitely generated $\mathbb{Z}$-graded $S$-module.
	\begin{enumerate}[(a)]
		\item Then $\beta_{i,j}^S(M)=(\Tor_{i}^S(M,\mathbb{K}))_j$ is called $(i,j)^{th}$ \emph{graded Betti numbers} of $M.$
		
		\item The \emph{regularity} of $M$, denoted as $\reg(M)$, is defined as
		\begin{equation*}
			\reg(M)=\max\{j-i:\beta_{i,j}^S(M)\neq 0\}.
		\end{equation*}
		
		\item If there is a integer $d$ such that $\beta_{i,i+b}=0$ for all $i$ and for all $b\neq d$, then we say that the module $M$ has \emph{linear resolution}.
		
		\item The module $M$ is said to be \emph{sequentially Cohen-Macaulay} if there exists a finite sequence
		\begin{equation*}
			0=M_0\subset M_1 \subset \cdots \subset M_r=M
		\end{equation*}
		of graded $S$-submodules of $M$ so that quotient module ${M_i}/{M_{i-1}}$ is Cohen-Macaulay for all $i\in [r]$, and $\dim(M_i/M_{i-1})<\dim(M_{i+1}/M_i)$ for all $i\in [r-1].$
	\end{enumerate}
\end{defn}

A simplicial complex $\Delta$ is said to be \emph{Cohen-Macaulay} (resp. \emph{sequentially Cohen-Macaulay}) over $\mathbb{K}$ if $S/I_{\Delta}$ is Cohen-Macaulay (resp. sequentially Cohen-Macaulay) ring. Further, a hypergraph $\mathcal{H}$ is said to be \emph{vertex decomposable} (resp. \emph{shellable, sequentially Cohen-Macaulay}) if its independence complex $\Delta(\mathcal{H})$ is vertex decomposable (resp. shellable, sequentially Cohen-Macaulay). The following implications for a hypergraph are known:
$$ \text{vertex decomposable} \Rightarrow \text{shellable} \Rightarrow \text{sequentially Cohen-Macaulay}. $$

\begin{defn}
	A graded ideal $I$ of $S$ is called \emph{componentwise linear} if $I_{<j>}$ has a linear resolution for all $j$, where $I_{<j>}$ is the ideal generated by all elements of $I$ of degree $j.$
\end{defn}

In general, it is not easy to find ideals which are componentwise linear. However,  there is a smaller class of componentwise linear ideals consisting of ideals with linear quotients.
\begin{defn}
	Let $I$ be a monomial ideal of $S.$ Then $I$ is said to have \emph{linear quotients} if there is an ordering $u_1,\dots,u_r$ of minimal generators of $I$ such that $\langle u_1,\dots,u_{i-1}\rangle:\langle u_i\rangle$ is generated by a subset of $\{x_1,\dots,x_n\}$
	for all $i.$
\end{defn}

Now we introduce the notion of  the $k$th \emph{symbolic power} of a squarefree monomial ideal.

\begin{defn}
	Let $I$ be a squarefree monomial ideal in $S$ with irredundant primary decomposition $I=\wp_1\cap\cdots\cap\wp_t$, where $\wp_i$ is an ideal generated by some variables in $S.$ Then for $k \in \mathbb{N}_{>0}$, the $k$th \emph{symbolic power}  $I^{(k)}$ of $I$, is defined as follows:
	\[
		I^{(k)}=\wp_1^{k}\cap\cdots\cap\wp_t^{k}.
	\]
\end{defn}

Now we recall the definition of polarization. It is a very useful tool to
convert a monomial ideal in to a squarefree monomial ideal.

\begin{defn}
	Let $I\subset S$ be a monomial ideal with generators $u_1,\dots,u_m$, where $u_i=\prod_{j=1}^{n}x_j^{a_{ij}}.$ For $j\in [n]$, set 
	$a_j=\max\{a_{ij}:i\in [m]\}.$ 
	Consider a polynomial ring $$T=\mathbb{K}[x_{11},\dots,x_{1a_1},x_{21},\dots,x_{2a_2},\dots,x_{n1},\dots,x_{na_n}].$$ 
	Then the squarefree monomial ideal $\widetilde{I}$ in $T$ generated by squarefree monomials $w_1,\dots,w_m$, where $w_i=\displaystyle\prod_{j=1}^{n}\prod_{k=1}^{a_{ij}}x_{jk}$, is called \emph{polarization} of $I.$
\end{defn}

\section{Powers of vertex cover ideal of a simplicial tree}

In this section, we study the minimal free resolution of powers of the cover ideal $J(\Delta)$ of a simplicial tree $\Delta.$ We prove that $J(\Delta)^k$ is componentwise linear for all $k \in \mathbb{N}_{>0}.$ In order to prove this result, we introduce some basic notion.

\begin{rmk}\label{R-Forest}\cite[Corollary 3.4]{Herzog2006SGV}
	A simplicial complex $\Delta$ has a good leaf order if and only if $\Delta$ is a forest. In particular, every forest has a good leaf.
\end{rmk}

\begin{defn}\label{D-GoodVertices} 
	For $t\ge 2$, let $\Delta=\langle F_1,\dots,F_t\rangle$ be a simplicial tree. Suppose $F_1$ is a good leaf of $\Delta.$ By definition of good leaf, there exists an ordering $i_1,\dots, i_{t-1}$ such that 
	$$F_1\cap F_{i_1}\supseteq F_1\cap F_{i_2}\supseteq\cdots\supseteq F_1\cap F_{i_{t-1}}.$$ 
	Now, let $g \in [t-1]$ be the largest integer such that $F_1\cap F_{i_g} \neq \emptyset.$ For $l\in [g]$, suppose that 
	$$F_1\cap F_{i_l} = \{x_{j_0},\dots,x_{j_1},\dots,x_{j_2},\dots,x_{j_{g-l}},\dots,x_{j_{g+1-l}}\}.$$ The finite sequence $(x_{j_0},\dots,x_{j_1},\dots,x_{j_2},\dots,x_{j_{g-1}},\dots,x_{j_g})$ is called \emph{a sequence of good vertices} of good leaf $F_1$ in $\Delta.$
\end{defn}	

In the following example, we see that sequence of good vertices of a leaf may not be unique.\\

\begin{minipage}{\linewidth}
\hspace{-0.5cm}	
	\begin{minipage}{0.7\linewidth}
		\begin{exmp}\label{E1}
		Let $\Delta$ be the simplicial tree as shown in Figure \ref{Fig 1} with facets 
		$F_1 = \{x_1,x_2,x_3,x_4\}$, $F_2 = \{x_1,x_2,x_3,x_5\}$, $F_3 = \{x_1,x_5,x_6\}$ and $F_4 = \{x_6,x_7,x_8\}.$
		We have 
		$$F_1\cap F_2 = \{x_1,x_2,x_3\}\supseteq F_1\cap F_3 = \{x_1\}\supseteq F_1\cap F_4 = \emptyset.$$
		Hence, $F_1$ is a good leaf of $\Delta.$ Also note that $(x_1,x_2,x_3)$ and $(x_1,x_3,x_2)$ both are sequences of good vertices of good leaf $F_1.$
	\end{exmp}
	\end{minipage}
		\begin{minipage}{0.30\linewidth}
\begin{figure}[H]
	\begin{tikzpicture}[scale=0.6][line cap=round,line join=round,>=triangle 45,x=1.0cm,y=1.0cm]
		\fill[lightgray] (0,0) -- (1,0.7) -- (1,2) -- cycle;
		\fill[lightgray] (0,0) -- (2,0) -- (1,2) -- (1,0.7) -- cycle;
		\fill[lightgray] (2,0) -- (1,2) -- (1,0.7) -- (3,2) -- cycle;
		\fill[lightgray] (1,2) -- (1,0.7) -- (3,2) -- cycle;
		\fill[lightgray] (2,0) -- (3,2) -- (4,0) -- cycle;
		\fill[lightgray] (4,0) -- (5,2) -- (6,0) -- cycle;
		\draw (0,0)-- (2,0);
		\draw (2,0)-- (1,2);
		\draw (1,2) -- (0,0);
		\draw (1,0.7)-- (0,0);
		\draw (1,0.7)-- (2,0);
		\draw (1,0.7)-- (1,2);
		\draw (1,0.7)-- (3,2);
		\draw (2,0)-- (3,2);
		\draw (1,2)-- (3,2);
		\draw (3,2)-- (4,0);
		\draw (2,0)-- (4,0);
		\draw (4,0)-- (5,2);
		\draw (4,0)-- (6,0);
		\draw (5,2)-- (6,0);
		
		\begin{scriptsize}
			\fill  (0,0) circle (1.5pt);
			\draw[below] (0,0) node {$x_4$};
			\fill  (1,2) circle (1.5pt);
			\draw[above] (1,2) node {$x_2$};
			\fill  (1,0.7) circle (1.5pt);
			\draw[below] (1,0.7) node {$x_3$};
			\fill  (2,0) circle (1.5pt);
			\draw[below] (2,0) node {$x_1$};
			\fill  (3,2) circle (1.5pt);
			\draw[above] (3,2) node {$x_5$};
			\fill  (4,0) circle (1.5pt);
			\draw[below] (4,0) node {$x_6$};
			\fill  (5,2) circle (1.5pt);
			\draw[above] (5,2) node {$x_7$};
			\fill  (6,0) circle (1.5pt);
			\draw[below] (6,0) node {$x_8$};
		\end{scriptsize}
	\end{tikzpicture}
	\caption{}
	\label{Fig 1}
\end{figure}
\end{minipage}
\end{minipage}
\\

To investigate symbolic powers of the cover ideal $J(\mathcal{H})$ of a hypergraph $\mathcal{H}$, we construct a new hypergraph from $\mathcal{H}$ such that the polarization of symbolic powers of $J(\mathcal{H})$  is equal to the cover ideal of the newly constructed hypergraph.

\noindent 
\textbf{Construction}.
Let $\mathcal{H}$ be a hypergraph with the vertex set 
$\mathcal{V}(\mathcal{H}) = \{x_1,\dots ,x_n\}$ and the edge set $\mathcal{E}(\mathcal{H}) = \{F_1,\dots ,F_t\}.$ Let $k\in \mathbb{N}_{>0}.$ For $\mathbf{f} =(f_1,\dots,f_b) \in [k]^{b}$, let $|\mathbf{f}|$ denotes the sum $\sum_{p=1}^{b} f_p.$ Then for an edge $F=\{x_{j_1},\dots ,x_{j_b}\}$, we define a hypergraph $F(k)$ with the vertex set
$$\mathcal{V}(F(k)) =\{x_{j_p,f}: p \in [b], f\in [k]\}$$ 
and the edge set 
$$\mathcal{E}(F(k)) =\{\{x_{j_1,f_1},\dots ,x_{j_b,f_b}\}:|\mathbf{f}| \le k+b-1\}.$$
By convention, we set $F(0)$ to be an isolated hypergraph on vertices $\mathcal{V}(F(0))=\{x_{j_1,1},\dots ,x_{j_a,1}\}.$ For 
$\mathbf{k}_t = (k_1,\dots,k_t)\in \mathbb{N}^t$, we define a hypergraph $\mathcal{H}(\mathbf{k}_t) = \mathcal{H}(k_1,\dots,k_t)$ with the vertex set $\mathcal{V}(\mathcal{H}(\mathbf{k}_t)) = \bigcup\limits_{i=1}^{t} \mathcal{V}(F_i(k_i))$ and the edge set 
$\mathcal{E}(\mathcal{H}(\mathbf{k}_t)) = \bigcup\limits_{i=1}^{t} \mathcal{E}((F_i(k_i))).$ In particular, if $k_i=k$ for all $i \in [t]$, then we write $\mathcal{H}(\mathbf{k})$ for the hypergraph $\mathcal{H}(\bk_t).$
\\\\
We illustrate above construction with help of the following example.
\begin{exmp}\label{E2}
	Let us consider the hypergraph $\mathcal{H}$ on the vertex set $\mathcal{V}(\mathcal{H}) = \{x_1,\dots,x_6\}$ and the edge set
	$\mathcal{E}(\mathcal{H}) = \{F_1,F_2,F_3\}$, where $F_1 = \{x_1,x_2,x_3\}$, $F_2 = \{x_3,x_4,x_5\}$, and $F_3 = \{x_4,x_5,x_6\}.$
	For $\mathbf{k}_3=(2,1,3)\in \mathbb{N}^3$, $\mathcal{H}(\mathbf{k}_3)$ is the hypergraph with the vertex set
    \begin{align*}
	    \mathcal{V}(\mathcal{H}(\mathbf{k}_3)) & = \{x_{1,1},x_{1,2},x_{2,1},x_{2,2},x_{3,1},x_{3,2},x_{4,1},x_{4,2},x_{4,3},x_{5,1},x_{5,2},x_{5,3},x_{6,1},x_{6,2},x_{6,3}\},
    \end{align*} 
    and the edge set\\
    $\mathcal{E}(\mathcal{H}(\mathbf{k}_3))=\left\{\begin{array}{ll}
        &\{x_{1,1},x_{2,1},x_{3,1}\},\{x_{1,2},x_{2,1},x_{3,1}\},\{x_{1,1},x_{2,2},x_{3,1}\},\{x_{1,1},x_{2,1},x_{3,2}\}, \nonumber \\
        &  \{x_{3,1},x_{4,1},x_{5,1}\},\{x_{4,1},x_{5,1},x_{6,1}\},\{x_{4,2},x_{5,1},x_{6,1}\},\{x_{4,1},x_{5,2},x_{6,1}\}, \nonumber \\ 
        &  \{x_{4,1},x_{5,1},x_{6,2}\},\{x_{4,3},x_{5,1},x_{6,1}\},\{x_{4,1},x_{5,3},x_{6,1}\},\{x_{4,1},x_{5,1},x_{6,3}\},\nonumber \\
        & \{x_{4,2},x_{5,2},x_{6,1}\},\{x_{4,2},x_{5,1},x_{6,2}\},\{x_{4,1},x_{5,2},x_{6,2}\} 
        \end{array}\right\}.$
\end{exmp}
 
The following lemma establishes a relation between the cover ideal $J(\mathcal{H}(\mathbf{k}))$ and the $k$th symbolic power of the cover ideal $J(\mathcal{H}).$
	  
\begin{lem}\label{L-Polarisation}
	Let $\mathcal{H}$ be a hypergraph with the vertex set $\mathcal{V}(\mathcal{H})=\{x_1,\dots,x_n\}$ and the edge set $\mathcal{E}(\mathcal{H})=\{F_1,\dots,F_t\}.$ Then 
	$J(\mathcal{H}(\mathbf{k})) = \widetilde{J(\mathcal{H})^{(k)}}$, where $k\in\mathbb{N}_{>0}.$
\end{lem}

\begin{proof}
    By using \cite[Proposition 2.3]{faridi}, we have 
    $$
    \widetilde{J(\mathcal{H})^{(k)}} = \widetilde{\wp_{F_1}^k} \cap \cdots \cap \widetilde{\wp_{F_t}^k}.
    $$
    Now, by \cite[Proposition 2.5]{faridi},
    \begin{align*}
    	\widetilde{\wp_F^k} & =\bigcap\limits_{\Sigma f_j\le k+a-1}(x_{j_1,f_1},\dots,x_{j_a,f_a}), \text{ where } 
    	F =\{x_{j_1},\dots,x_{j_a}\}\subset \mathcal{E}(\mathcal{H}).
    \end{align*}
    Thus, we have the equality $\widetilde{J(\mathcal{H})^{(k)}} = J(\mathcal{H}(\mathbf{k})).$		                                             
\end{proof}

The following theorem is the main result of the paper.

\begin{thm}\label{Th-MainTheorem}
	Let $\mathbf{k}_t \in \mathbb{N}^t_{>0}$ and $\Delta$ be a simplicial tree. Then $\mathcal{H}(\mathbf{k}_t)$ is a vertex decomposable hypergraph, where $\mathcal{H} =\mathcal{H}(\Delta).$
\end{thm}

We prove this theorem in the following sequence of lemmas. First we start with the following fact.

\begin{rmk}\label{R-Edge}
	 Let $\mathcal{H}$ be a hypergraph on the vertex set $\mathcal{V}(\mathcal{H})=\{x_1,\dots,x_n\}$ and the edge set $\mathcal{E}(\mathcal{H})=\{\{x_1,\dots,x_n\}\}.$ For $k\in \mathbb{N}$, let $\mathcal{D}=\Delta(\mathcal{H}(\bk))$ and $W\in \link_{\mathcal{D}}(x_{1,1}).$
	 Since $\{x_{1,1},\dots,x_{n,1}\}$ is an edge in $\mathcal{H}(\bk)$ and $W\cup\{x_{1,1}\}\in \mathcal{D}$, it follows that $\{x_{j,1},\dots,x_{j,k}\}\not\subset W$ for some $2 \le j \le n.$
\end{rmk}

\begin{lem}\label{L-Simplex}
    With the notation as in Remark \ref{R-Edge}, $\mathcal{H(\bk)}$ is vertex decomposable.
\end{lem}	

\begin{proof} 
	We prove the lemma using induction on $N=k+n.$ If $n=1$ or $k =0$, then $\mathcal{H}(\bk)$ is an isolated hypergraph, and hence vertex decomposable. In particular, the proof follows if $N=1.$ 
	
	Now, let us assume that $N>1$ and $n>1.$ First, we prove that $x_{1,1}$ is a shedding vertex of hypergraph $\mathcal{H}(\bk).$ Let $\mathcal{D}=\Delta(\mathcal{H}(\bk))$ and $W\in \link_{\mathcal{D}}(x_{1,1}).$ By Remark \ref{R-Edge}, there exists $2\le j\le n$ such that $\{x_{j,1},\dots,x_{j,k}\}\not\subset W.$
	Suppose $f$ is the largest positive integer such that $x_{j,f} \notin W.$ We claim that $W \cup \{x_{j,f}\} \in\del_{\mathcal{D}}(x_{1,1}).$ Since $x_{1,1} \notin W$, it suffices to prove that $W \cup \{x_{j,f}\} \in \mathcal{D}.$ On the contrary, suppose that $W \cup \{x_{j,f}\} \notin \mathcal{D}.$ Then there exists $\mathbf{f} =(f_1,\dots,f_n)\in [k]^n$ with $|\mathbf{f}| \le k+n-1$ such that $\{x_{1,f_1},\dots,x_{n,f_n}\} \subset W\cup\{x_{j,f}\}.$ This implies that $f = f_j$ and $x_{1,f_1} \in W.$ Since $x_{1,1} \notin W$, we have $f_1 \ge 2.$ This gives 
	\begin{equation}\label{EQ1}
		1+f_2+ \cdots +(f_j+1)+ \cdots +f_n \le |\mathbf{f}| \le k+n-1.
	\end{equation}
	Now, using facts $f_1\ge 2$, $f_i\ge 1$ for all $i$, and $|\mathbf{f}| \le k+n-1$, we get $f_j = f < k.$ This gives 
	$$\{x_{1,1},x_{2,f_2},\dots,x_{j,f_j+1},\dots,x_{n,f_n}\}\subset W\cup\{x_{1,1}\},$$ and hence by Equation \eqref{EQ1} we have $\{x_{1,1},x_{2,f_2},\dots,x_{j,f_j+1},\dots,x_{n,f_n}\} \in \mathcal{E}(\mathcal{H}(\bk)).$ 
	This gives a contradition to the fact $W\cup\{x_{1,1}\} \in \mathcal{D}.$ This proves that $W\cup\{x_{j,f}\} \in \mathcal{D}$, and hence $x_{1,1}$ is a shedding vertex of $\mathcal{H}(\bk).$ 
	 
	Now $\mathcal{H}(\bk)\setminus \{x_{1,1}\}\simeq \mathcal{H}(\bk-{\bf 1})$ and $\mathcal{H}(\bk)/{x_{1,1}}=\mathcal{H}^\prime(\bk)$, where $\mathcal{H}^\prime$ is the hypergraph with the vertex set $\mathcal{V}(\mathcal{H}^\prime)=\{x_2,\dots,x_n\}$ and the edge set  $\mathcal{E}(\mathcal{H}^\prime)=\{\{x_2,\dots,x_n\}\}.$ By induction on $N$, $\mathcal{H}(\bk)\setminus \{x_{1,1}\}$ and $\mathcal{H}(\bk)/{x_{1,1}}$ are vertex decomposable. Therefore, $\mathcal{H}(\bk)$ is vertex decomposable.
\end{proof}

Now, we fix some notation for the rest of the section.

\begin{set}\label{Set-up}
	For $t \ge2$, let $\Delta = \langle F_1,\dots,F_t \rangle$ be a simplicial tree on the vertex set $\mathcal{V}=\{x_1,\dots,x_n\}$ with good leaf order $F_1,\dots,F_t.$ For $i\in [t-1]$, suppose $G_i$ is a branch of $F_i$ in $\Delta_i = \langle F_i,\dots,F_t\rangle.$ Since $F_1$ is a good leaf of $\Delta$, there exists an ordering $i_1,\dots,i_t$ such that 
	$$F_1\cap F_{i_1}\supset F_1\cap F_{i_2}\supset\cdots\supset F_1\cap F_{i_t}.$$ 
	Let $g$ be the largest integer such that $F_1\cap F_{i_g} \neq \emptyset.$
	Consider the hypergraph $\mathcal{H} = \mathcal{H}(\Delta).$ 
	Let $\mathbf{k}_t = (k_1,\dots,k_t)\in \mathbb{N}^t_{>0}$ and 
	$\mathcal{L} =P_0P_1 P_2\dots$ be an infinite string with $P_i\in\{L,D\}.$ Suppose that 
	$\mathcal{L}_s =P_0P_1 P_2\dots P_s$ be its $s$th partial string. We define, recursively, a sequence of triplet $\left(\overline{\mathcal{H}}_{\mathcal{L}_s},A_s,B_s\right)_{s=0}^{\infty}$, where 
	\begin{enumerate}
		\item $\overline{\mathcal{H}}_{\mathcal{L}_s}$ is a hypergraph; and
		
		\item $A_s$ and $B_s$ are sets of vertices.
	\end{enumerate} 
	Define
	\begin{align*}
		\left(\overline{\mathcal{H}}_{\mathcal{L}_0},A_0,B_0\right) & = \left(\mathcal{H}(\mathbf{k}_t),\emptyset,\emptyset\right),
	\end{align*}
    and
	$$
	\begin{array}{ccl}
		\left(\overline{\mathcal{H}}_{\mathcal{L}_1},A_1,B_1\right)&=&
		\left\{\begin{array}{lll}
			\left(\mathcal{H}(\mathbf{k}_t)/{x_{u_1,c_1}},\{x_{u_1}\},\emptyset\right) &~\mbox{if}~P_1 = L;\\
			\left(\mathcal{H}(\mathbf{k}_t)\setminus {x_{u_1,c_1}},\emptyset,\{x_{u_1}\}\right) &~\mbox{otherwise},
		\end{array}\right.
	\end{array}
	$$
	where $x_{u_1} \in F_1\cap F_{i_g}$ and $c_1 =1.$ Let $s \ge 1$ be a integer such that there exists a sequence of vertices $(x_{u_q,c_q})_{q=1}^{s}$ with
	$$
	\begin{array}{ccl}
		\overline{\mathcal{H}}_{\mathcal{L}_q}&=&
		\left\{\begin{array}{lll}
			\overline{\mathcal{H}}_{\mathcal{L}_{q-1}}/{x_{u_q,c_q} }&~\mbox{if}~P_q = L;\\
			\overline{\mathcal{H}}_{\mathcal{L}_{q-1}} \setminus {x_{u_q,c_q} } &~\mbox{otherwise},
		\end{array}\right.
	\end{array}
	$$
	$A_{q} = \{x_{u_p}: P_p = L, p \in [q]\}$ and 
	$B_{q} = \{x_{u_p}: P_p = D, p \in [q]\}$ for all $q \in [s].$
	For $i \in [t]$, let 
	$$k^{(s)}_i=\max\{0,k_i-d^{(s)}_i\}, \text{ where } 
	d^{(s)}_i = -|F_i\cap A_{s}|+\sum_{x_{u_m} \in F_i\cap A_{s}} c_m.$$ 
	For $i \in [t-1]$, consider the following conditions:
	\begin{enumerate}[(i)]
		\item if $F_i\setminus{A_{s}}=\{x_{j_1},\dots,x_{j_b}\}$, then  $\{x_{j_1,f_1},\dots,x_{j_b,f_b}\}$ is an edge in $(\overline{\mathcal{H}}_{\mathcal{L}_{s}})^\circ$ for some $\mathbf{f}=(f_1,\dots,f_b)$ in $[k_i]^b$ with $f_p > c_m$ if $x_{j_p} = x_{u_m} \in B_{s}\setminus A_{s}$ and $|\mathbf{f}| \le k_i^{(s)}+b-1$; and
		
		\item $F_i\cap G_i\not{\subset}A_{s}.$
	\end{enumerate}
	Let $U_{s}$ be the set of all $i\in [t-1]$ satisfying above conditions. Assume that $U_{s}$ is non-empty. Let $\ell_{s}$ be the smallest element of $U_{s}$ and $(x_{w_1},\dots,x_{w_r})$ be a sequence of good vertices of $F_{\ell_{s}}$ in $\Delta_{\ell_{s}}.$ Now, set $j_0=\min\{j:x_{w_j}\not\in A_{s}\}$ and take $x_{u_{s+1}} = x_{w_{j_0}}.$ We define
	$$
	\begin{array}{ccl}
		\left(\overline{\mathcal{H}}_{\mathcal{L}_{s+1}},A_{s+1},B_{s+1}\right)&=&
		\left\{\begin{array}{lll}
			\left(\overline{\mathcal{H}}_{\mathcal{L}_{s}}
			/{x_{u_{s+1},c_{s+1}}},
			A_{s}\cup
			\{x_{u_{s+1}}\},B_{s}\right)&~\mbox{if}~P_s=L\\
			\left(\overline{\mathcal{H}}_{\mathcal{L}_{s}}
			\setminus x_{u_{s+1},c_{s+1}},A_{s},
			B_{s}\cup \{x_{u_{s+1}}\}\right)&~\mbox{otherwise},
		\end{array}\right.
	\end{array}
	$$
	where $c_{s+1}$ is the smallest positive integer such that $x_{u_{s+1},c_{s+1}}\in \mathcal{V}((\overline{\mathcal{H}}_{\mathcal{L}_{s}})^\circ).$
	On the other hand, if $U_{s} =\emptyset$, then we define 
	$\left(\overline{\mathcal{H}}_{\mathcal{L}_{q}},A_{q},B_{q}\right) = \left(\overline{\mathcal{H}}_{\mathcal{L}_{s}},A_{s},B_{s}\right)$ for all $q \ge s.$
\end{set}

\begin{rmk}\label{R-BasicRemark1}
	With notation as in Set-up \ref{Set-up}, we observe the following.
	\begin{enumerate}[(a)]
		\item Since $\mathcal{V}(\mathcal{H}(\bk_t))$ is finite,
		it follows from Set-up \ref{Set-up} that there exists unique integer 
		$\alpha (\mathcal{L})$ such that
		$$
		\begin{array}{ccl}
			\left(\overline{\mathcal{H}}_{\mathcal{L}_{s}},A_s,B_s\right)&=&
			\left\{\begin{array}{lll}
				\left((\overline{\mathcal{H}}_{\mathcal{L}_{s-1}})/{x_{u_s,c_s}},A_{s-1}\cup\{x_{u_s}\},B_{s-1}\right)&~\mbox{if}~P_s=L\\
				\left((\overline{\mathcal{H}}_{\mathcal{L}_{s-1}})\setminus x_{u_s,c_s},A_{s-1},B_{s-1}\cup\{x_{u_s}\}\right)&~\mbox{otherwise}
			\end{array}\right.
		\end{array}
		$$
		for all $s \in [\alpha(\mathcal{L})]$ and $$\left(\overline{\mathcal{H}}_{\mathcal{L}_{s}},A_s,B_s\right) =\left(\overline{\mathcal{H}}_{\mathcal{L}_{\alpha(\mathcal{L})}},
		A_{\alpha(\mathcal{L})},B_{\alpha(\mathcal{L})}\right)$$ 
		for all 
		$s \ge \alpha(\mathcal{L}).$ Further, $U_s \ne \emptyset$ for all 
		$s \in[\alpha(\mathcal{L})-1]$ and $U_{\alpha(\mathcal{L})} =\emptyset.$
		
		\item Let $s \in[\alpha(\mathcal{L})-1]\cup\{0\}.$ Suppose that 
		$F_{\ell_s}\setminus A_{s} =\{x_{u_{s+1}}=x_{h_1},\dots,x_{h_a}\}$, where $$
		\begin{array}{ccl}
			l_{s}& =&
			\left\{\begin{array}{lll}
				\min\{i:i\in U_s\}&~\mbox{if}~ s\in[\alpha(\mathcal{L})-1]\\
				1&~\mbox{if}~ s =0.
			\end{array}\right.
		\end{array}
		$$
		Then $\{x_{h_1,e_1},\dots,x_{h_a,e_a}\}$ is an edge in $(\overline{\mathcal{H}}_{\mathcal{L}_s})^\circ$ for some $\mathbf{e} =(e_1,\dots,e_a)\in [k_{\ell_s}]^a$
		with $e_p > c_m$ if $x_{h_p} = x_{u_m} \in B_{\ell_s}\setminus A_{\ell_s}$ and $|\mathbf{e}| \le k_{\ell_s}^{(s)}+a-1.$
    \end{enumerate}	
\end{rmk}

\begin{rmk}\label{R-BasicRemark2}
    With notation as in Set-up \ref{Set-up} and Remark \ref{R-BasicRemark1}, let $s \in[\alpha(\mathcal{L})-1].$
    \begin{enumerate}[(a)]
		\item Let $x_{u_{s+1}} = x_{u_q}$ for some $q \le s$ with $P_q = D.$ Since
        $x_{u_{s+1},c_{s+1}} \in \mathcal{V}((\overline{\mathcal{H}}_{\mathcal{L}_{s}})^\circ)$ and 
        $$\mathcal{V}((\overline{\mathcal{H}}_{\mathcal{L}_{s}})^\circ) \subset \mathcal{V}((\overline{\mathcal{H}}_{\mathcal{L}_{q-1}})^\circ),$$ we get 
        $c_{s+1} > c_m.$
		
		\item Let $E \in \mathcal{E}(\overline{\mathcal{H}}_{\mathcal{L}_{s}})$ with $x_{u_{s+1},c_{s+1}} \in E.$ Since $x_{u_{s+1},c_{s+1}} \in \mathcal{V}((\overline{\mathcal{H}}_{\mathcal{L}_{s}})^\circ)$, we get
		$E \in \mathcal{E}((\overline{\mathcal{H}}_{\mathcal{L}_{s}})^\circ).$
	\end{enumerate}
\end{rmk}

With notation as in Set-up \ref{Set-up} and Remark \ref{R-BasicRemark1}, we have defined $k^{(s)}_i$ for 
$s \in[\alpha(\mathcal{L})].$ For $s \ge \alpha(\mathcal{L})$, we define 
$k^{(s)}_i=k^{(\alpha(\mathcal{L}))}_i.$  
Further, a non-empty subset $E =\{x_{j_1,f_1},\dots,x_{j_b,f_b}\}$ of
$\mathcal{V}(\mathcal{H}(\bk_t))$ is said to be \emph{constructible} in $\overline{\mathcal{H}}_{\mathcal{L}_s}$ if there exists $i \in [t]$ such that the following conditions hold:    
\begin{enumerate}[(a)]
	\item $\{x_{j_1},\dots,x_{j_b}\} = F_i\setminus A_s$; and
	
	\item $\mathbf{f} =(f_1,\dots,f_b) \in [k_i]^b$ with $f_p >c_m$ if 
	$x_{j_p} = x_{u_m} \in B_s\setminus A_s$ and $|\mathbf{f}| \le k_i^{(s)}+b-1.$
\end{enumerate}
 
\vspace{0.3cm}
The following is an illustrative example for Set-up \ref{Set-up}.\\

\begin{minipage}{\linewidth}
	\hspace{-0.5cm}	
	\begin{minipage}{0.6\linewidth}
		\begin{exmp}\label{E-illustrative}
			Let $\Delta$ be a simplicial tree as shown in Figure \ref{Fig 2} with facets $F_1 = \{x_1,x_2,x_3\}$, $F_2 = \{x_1,x_4,x_5,x_6\}$ and 
			$F_3 = \{x_5,x_6,x_7,x_8\}.$ With notation as in Set-up \ref{Set-up}, let $\mathcal{L} = P_0 P_1 P_2 \dots$ be an infinite string such that $P_2 =D$ and $P_1 =P_3 =P_4 =L.$ Let $\mathbf{k}_3=(1,2,2)\in \mathbb{N}^3_{>0}$ and $\mathcal{H} =\mathcal{H}(\Delta)$ be the hypergraph associated to $\Delta.$
			We calculate the sequence $(\overline{\mathcal{H}}_{\mathcal{L}_s})_{s=0}^{\infty}.$ By definition, $\overline{\mathcal{H}}_{\mathcal{L}_0}=\mathcal{H}(\bk_3).$ Since $x_{u_1} = x_1$, we have 
		\end{exmp}
	\end{minipage}
    \hspace{1cm}
	\begin{minipage}{0.30\linewidth}
		\begin{figure}[H]
			\begin{tikzpicture}[scale=0.6][line cap=round,line join=round,>=triangle 45,x=1.0cm,y=1.0cm]
				\fill[lightgray] (0,0) -- (1,2) -- (2,0) -- cycle;
				\fill[lightgray] (5,2) -- (3.86,0.6) -- (5,0) -- cycle;
				\fill[lightgray] (2,0) -- (5,0) -- (3.86,0.6) -- (5,2) -- cycle;
				\fill[lightgray] (5,2) -- (5,0) -- (8,0) -- (6.14,0.6) -- cycle;
				\fill[lightgray] (6.14,0.6) -- (8,0) -- (5,2) -- cycle;
				\draw (0,0)-- (1,2);
				\draw (1,2)-- (2,0);
				\draw (2,0)-- (0,0);
				\draw (2,0)-- (3.86,0.6);
				\draw (3.86,0.6)-- (5,0);
				\draw (5,0)-- (2,0);
				\draw (5,0)-- (5,2);
				\draw (5,2)-- (3.86,0.6);
				\draw (5,2)-- (2,0);
				\draw (6.14,0.6)-- (5,0);
				\draw (6.14,0.6)-- (8,0);
				\draw (6.14,0.6)-- (5,2);
				\draw (5,2)-- (8,0);
				\draw (5,0)-- (8,0);
				
				\begin{scriptsize}
					\fill  (0,0) circle (1.5pt);
					\draw[below] (0,0) node {$x_2$};
					\fill  (2,0) circle (1.5pt);
					\draw[below] (2,0) node {$x_1$};
					\fill  (1,2) circle (1.5pt);
					\draw[above] (1,2) node {$x_3$};
					\fill  (3.86,0.6) circle (1.5pt);
					\draw[right] (3.86,0.6) node {$x_4$};
					\fill  (5,0) circle (1.5pt);
					\draw[below] (5,0) node {$x_6$};
					\fill  (5,2) circle (1.5pt);
					\draw[above] (5,2) node {$x_5$};
					\fill  (6.14,0.6) circle (1.5pt);
					\draw[left] (6.14,0.6) node {$x_7$};
					\fill (8,0) circle (1.5pt);
					\draw[below] (8,0) node {$x_8$};
				\end{scriptsize}
			\end{tikzpicture}
			\caption{}
			\label{Fig 2}
		\end{figure}
	\end{minipage}
\end{minipage}
\\\\

    $\mathcal{E}(\overline{\mathcal{H}}_{\mathcal{L}_1})=
    \left\{\begin{array}{ll}
	    &\{x_{2,1},x_{3,1}\},\{x_{4,1},x_{5,1},x_{6,1}\},
	    \{x_{4,2},x_{5,1},x_{6,1}\}, \{x_{4,1},x_{5,2},x_{6,1}\}, \\
	    &
	    \{x_{4,1},x_{5,1},x_{6,2}\},\{x_{5,1},x_{6,1},x_{7,1},x_{8,1}\},
	    \{x_{5,2},x_{6,1},x_{7,1},x_{8,1}\}, \\
	    &
	    \{x_{5,1},x_{6,2},x_{7,1},x_{8,1}\},\{x_{5,1},x_{6,1},x_{7,2},x_{8,1}\},
	    \{x_{5,1},x_{6,1},x_{7,1},x_{8,2}\}
    \end{array}\right\}.$\\\\
    By Set-up \ref{Set-up}, $A_1 =\{x_1\}$ and $B_1 =\emptyset.$
    We find that $\ell_1=2.$ Since $(x_5,x_6)$ is a sequence of good vertices of good leaf $F_2$ in $\Delta_2 =\langle F_2,F_3\rangle$, we have 
    $x_{u_2,c_2} =x_{5,1}.$ Hence $\overline{\mathcal{H}}_{\mathcal{L}_2}$ is a hypergraph with the edge set
    $$
    \mathcal{E}(\overline{\mathcal{H}}_{\mathcal{L}_2})=
    \Bigl\{\{x_{2,1},x_{3,1}\},\{x_{4,1},x_{5,2},x_{6,1}\},
    \{x_{5,2},x_{6,1},x_{7,1},x_{8,1}\}\Bigr\}.
    $$
    We have $A_2 =\{x_1\}$ and $B_2 =\{x_5\}.$
    Further, $\ell_2 =2$, we get $x_{u_3,c_3} =x_{5,2}.$ Hence $\overline{\mathcal{H}}_{\mathcal{L}_3}$ is a hypergraph with the edge set
    $$
    \mathcal{E}(\overline{\mathcal{H}}_{\mathcal{L}_3})=
    \Bigl\{\{x_{2,1},x_{3,1}\},\{x_{4,1},x_{6,1}\},\{x_{6,1},x_{7,1},x_{8,1}\}
    \Bigr\}.
    $$
    This gives us $A_3 =\{x_1,x_5\}$ and $B_3 =\{x_5\}.$ 
    Now, we have $\ell_3 =2$, and hence $x_{u_4,c_4} =x_{6,1}.$ Therefore, $\overline{\mathcal{H}}_{\mathcal{L}_4}$ is a hypergraph with the edge set
    $$
    \mathcal{E}(\overline{\mathcal{H}}_{\mathcal{L}_4})=
    \Bigl\{\{x_{2,1},x_{3,1}\},\{x_{4,1}\},\{x_{7,1},x_{8,1}\}\Bigr\}.
    $$
    As before, $A_4 =\{x_1,x_5,x_6\}$ and $B_4 =\{x_5\}.$ Further, we observe that $U_4 =\emptyset.$ Therefore, we have 
    $\overline{\mathcal{H}}_{\mathcal{L}_s}=\overline{\mathcal{H}}_{\mathcal{L}_4}$
    for all $s \ge 4.$
    Note that $x_{u_2} = x_{u_3}.$ Thus, it is possible that $x_{u_p} =x_{u_q}$ for some $p\neq q.$\\

With the notation of Example \ref{E-illustrative}, observe that $x_{1,2} \in \mathcal{V}(\overline{\mathcal{H}}_{\mathcal{L}_1})$ but 
$\{x_{1,2}\} \notin \mathcal{E}(\overline{\mathcal{H}}_{\mathcal{L}_1}).$ Also, observe that 
$x_{6,2} \in \mathcal{V}(\overline{\mathcal{H}}_{\mathcal{L}_2})$ but 
$\{x_{6,2}\} \notin \mathcal{E}(\overline{\mathcal{H}}_{\mathcal{L}_2}).$
Thus, the operations of contraction and deletion may produce such isolated vertices which do not belong to an edge.
We keep this observation in mind and give the following characterization of the edge set of the hypergraph  $\overline{\mathcal{H}}_{\mathcal{L}_s}.$

\begin{lem}\label{L-Constructible}
    With notation as in Set-up \ref{Set-up}, suppose that
    $E$ is an edge in $\overline{\mathcal{H}}_{\mathcal{L}_s}.$ Then $E$ is constructible in $\overline{\mathcal{H}}_{\mathcal{L}_s}.$ 
    Conversely, if $E$ is constructible in $\overline{\mathcal{H}}_{\mathcal{L}_s}$, then there exists an edge $E'$ in $\overline{\mathcal{H}}_{\mathcal{L}_s}$ such that $E'\subset E.$
\end{lem}

\begin{proof}
    We prove the result using induction on $s.$ 
    By construction of hypergraph $\mathcal{H}(\bk_t)$, the result holds for $s =0.$ Let 
    $s >0$ and the result holds for $s-1.$ If $s >\alpha(\mathcal{L})$, then $\overline{\mathcal{H}}_{\mathcal{L}_s} =\overline{\mathcal{H}}_{\mathcal{L}_{s-1}}$, and hence result follows by using induction on $s.$
    Now, suppose $s \le\alpha(\mathcal{L}).$ Then there are two possibilities.
    
    \noindent
    {\bf Case 1.} When $P_s=L.$ By Set-up \ref{Set-up}, we have $A_s=A_{s-1}\cup\{x_{u_s}\}$ and $B_s=B_{s-1}.$ 
    Now, let $E =\{x_{j_1,f_1},\dots,x_{j_b,f_b}\} \in \mathcal{E}(\overline{\mathcal{H}}_{\mathcal{L}_s}).$ Then either $E$ or $E\cup \{x_{u_s,c_s}\}$ is an edge in $\overline{\mathcal{H}}_{\mathcal{L}_{s-1}}.$
	If $E \in \mathcal{E}(\overline{\mathcal{H}}_{\mathcal{L}_{s-1}})$, then by induction on $s$, $E$ is constructible in $\overline{\mathcal{H}}_{\mathcal{L}_{s-1}}$, and hence there exists $i \in [t]$ such that $\{x_{j_1},\dots,x_{j_b}\} = F_i\setminus A_{s-1}$ and $\mathbf{f} =(f_1,\dots,f_b) \in [k_i]^b$ with 
	$$f_p>c_m ~\mbox{if}~ x_{j_p} = x_{u_m} \in B_{s-1}\setminus A_{s-1} ~\mbox{and}~ |\mathbf{f}| \le k_i^{(s-1)}+b-1.$$
	We claim that $x_{u_s} \notin F_i\setminus A_{s-1}.$ On the contrary, assume that $x_{u_s} \in F_i\setminus A_{s-1}$ and $x_{u_s} = x_{j_1}.$ Let $F =\{x_{j_1,c_s},x_{j_2,f_2},\dots,x_{j_b,f_b}\}.$ Now, by using $f_1\ge c_s$, we get
    $$|\mathbf{f}|-f_1+c_s \le k_i^{(s-1)}+b-1.$$
    Thus, by Remark \ref{R-BasicRemark2}(a), $F$ is constructible in $\overline{\mathcal{H}}_{\mathcal{L}_{s-1}}.$ Using induction on $s$, there exists an edge $F'\in \mathcal{E}(\overline{\mathcal{H}}_{\mathcal{L}_{s-1}})$ such that $F' \subset F.$ 
    Now, by Remark \ref{R-ContRemark}, there exists an edge $E' \in \mathcal{E}(\overline{\mathcal{H}}_{\mathcal{L}_s})$ such that 
    $E'\subset F'\setminus \{x_{u_s,c_s}\} \subset E.$ This implies that $E$ is not an edge in 
    $\overline{\mathcal{H}}_{\mathcal{L}_s}$, a contradiction. Thus, 
    $x_{u_s}\notin F_i\setminus A_{s-1}.$ Now, we have 
    $$F_i\setminus A_s = F_i\setminus A_{s-1} ~\mbox{and}~ F_i\cap A_s = F_i\cap A_{s-1}.$$ 
    Thus, we get $k_i^{(s)} = k_i^{(s-1)}.$ This implies that $|\mathbf{f}| \le k_i^{(s)}+b-1.$
    Further, if $x_{j_p} = x_{u_m} \in B_s\setminus A_s$, then $x_{j_p} = x_{u_m}\in B_{s-1}\setminus A_{s-1}$, and hence $f_p> c_m.$ Thus, $E$ is constructible in $\overline{\mathcal{H}}_{\mathcal{L}_s}.$
    On the other hand, if $E\cup \{x_{u_s,c_s}\} \in \mathcal{E}(\overline{\mathcal{H}}_{\mathcal{L}_{s-1}})$, then by induction on $s$, $E\cup \{x_{u_s,c_s}\}$ is constructible in $\overline{\mathcal{H}}_{\mathcal{L}_{s-1}}$, and hence there exists $i \in [t]$ such that $\{x_{j_1},\dots,x_{j_b},x_{u_s}\} = F_i\setminus A_{s-1}$ and $\mathbf{f} =(f_1,\dots,f_b) \in [k_i]^b$ with 
    $$f_p>c_m~ \mbox{if}~ x_{j_p} = x_{u_m} \in B_{s-1}\setminus A_{s-1} ~\mbox{and}~ |\mathbf{f}|+c_s \le k_i^{(s-1)}+b.$$ Thus, we have $\{x_{j_1},\dots,x_{j_b}\} = F_i\setminus A_s.$ Note that $k_i^{(s)} = k_i^{(s-1)}-c_s+1.$ Therefore, we get $|\mathbf{f}| \le k_i^{(s)}+b-1.$ Now, the fact $B_{s}\setminus A_{s} \subset B_{s-1}\setminus A_{s-1}$ implies that $E$ is constructible in $\overline{\mathcal{H}}_{\mathcal{L}_s}.$ 
	
	Conversely, let $E =\{x_{j_1,f_1},\dots,x_{j_b,f_b}\}$ is constructible in $\overline{\mathcal{H}}_{\mathcal{L}_s}.$ Then there exists $i \in [t]$ such that $\{x_{j_1},\dots,x_{j_b}\} = F_i\setminus A_s$ and $\mathbf{f} =(f_1,\dots,f_b) \in [k_i]^b$ with 
	$$f_p>c_m ~\mbox{if}~ x_{j_p} = x_{u_m} \in B_s\setminus A_s ~\mbox{and}~ |\mathbf{f}| \le k_i^{(s)}+b-1.$$ 
	If $x_{j_p} = x_{u_m}\in B_{s-1}\setminus A_{s-1}$, then $x_{j_p} \in B_s\setminus A_s$ as $x_{j_p} \neq x_{u_s}.$ Thus, we have $f_p > c_m$ if $x_{j_p} = x_{u_m} \in B_{s-1}\setminus A_{s-1}.$ Now, we have either $x_{u_s} \in F_i$ or $x_{u_s} \notin F_i.$ If $x_{u_s} \in F_i$, then we get $F_i\setminus A_{s-1} = \{x_{j_1},\dots,x_{j_b},x_{u_s}\}$ and $k_i^{(s-1)} = k_i^{(s)}+c_s-1.$ This implies that $$|\mathbf{f}|+c_s \le k_i^{(s-1)}+b.$$ By Remark \ref{R-BasicRemark2}(a), $E \cup \{x_{u_s,c_s}\}$ is constructible in $\overline{\mathcal{H}}_{\mathcal{L}_{s-1}}$, and hence by induction on $s$, there exists an edge $H \in \mathcal{E}(\overline{\mathcal{H}}_{\mathcal{L}_{s-1}})$ such that $H \subset E \cup \{x_{u_s,c_s}\}.$ Now, Remark \ref{R-ContRemark} implies that there exists an edge $E' \in \mathcal{E}(\overline{\mathcal{H}}_{\mathcal{L}_s})$ such that $E' \subset H\setminus \{x_{u_s,c_s}\} \subset E.$ On the other hand, if $x_{u_s} \notin F_i$, then 
	$$F_i\setminus A_{s-1} = F_i\setminus A_s ~\mbox{and}~ k_i^{(s-1)} = k_i^{(s)}.$$ 
	Now, $|\mathbf{f}| \le k_i^{(s-1)}+b-1$ implies that $E$ is constructible in $\overline{\mathcal{H}}_{\mathcal{L}_{s-1}}.$ By induction on $s$, there exists an edge $H \in \mathcal{E}(\overline{\mathcal{H}}_{\mathcal{L}_{s-1}})$ such that $H\subset E.$ Note that $x_{u_s,c_s}\notin H.$ Thus, by using Remark \ref{R-ContRemark}, there exists an edge $E' \in \mathcal{E}(\overline{\mathcal{H}}_{\mathcal{L}_s})$ such that $E'\subset H \subset E.$
    
    \noindent
    {\bf Case 2.} When $P_s=D.$ By Set-up \ref{Set-up}, $A_s=A_{s-1}$ and $B_s=B_{s-1}\cup\{x_{u_s}\}.$ Hence $k_i^{(s)} = k_i^{(s-1)}$ for all 
    $i \in [t].$ 
    Let $E =\{x_{j_1,f_1},\dots,x_{j_b,f_b}\}$ be an edge in $\overline{\mathcal{H}}_{\mathcal{L}_s}.$ Then $E \in \mathcal{E}(\overline{\mathcal{H}}_{\mathcal{L}_{s-1}})$ and $x_{u_s,c_s}\notin E.$ Thus, by induction on $s$, $E$ is constructible in $\overline{\mathcal{H}}_{\mathcal{L}_{s-1}}$, and hence there exists $i \in [t]$
    such that $\{x_{j_1},\dots,x_{j_b}\} = F_i\setminus A_{s-1}$ and $\mathbf{f}=(f_1,\dots,f_b) \in [k_i]^b$ with 
    $$f_p> c_m ~\mbox{if}~ x_{j_p} = x_{u_m} \in B_{s-1}\setminus A_{s-1} ~\mbox{and}~ |\mathbf{f}| \le k_i^{(s-1)}+b-1.$$
    Further, since $F_i\setminus A_{s-1} =F_i\setminus A_s$ and $k_i^{(s)} = k_i^{(s-1)}$, we have only to show that $f_p > c_s$ if $x_{j_p} =x_{u_s}.$ It follows from the fact that $x_{u_s,c_s} \notin E.$ Therefore, $E$ is constructible in $\overline{\mathcal{H}}_{\mathcal{L}_s}.$
    
    Conversely, let $E =\{x_{j_1,f_1},\dots,x_{j_b,f_b}\}$ is constructible in $\overline{\mathcal{H}}_{\mathcal{L}_s}.$ Then there exists $i \in [t]$ such that $\{x_{j_1},\dots,x_{j_b}\} = F_i\setminus A_s$ and $\mathbf{f} =(f_1,\dots,f_b) \in [k_i]^b$ with 
    $$f_p>c_m ~\mbox{if}~ x_{j_p} = x_{u_m} \in B_s\setminus A_s ~\mbox{and}~ |\mathbf{f}| \le k_i^{(s)}+b-1.$$ 
    If $x_{j_p} = x_{u_m}\in B_{s-1}\setminus A_{s-1}$, then $x_{j_p} \in B_s\setminus A_s$ implies that $f_p >c_m.$ Now, by using the fact that $k_i^{(s)} = k_i^{(s-1)}$, we get $E$ is constructible in $\overline{\mathcal{H}}_{\mathcal{L}_{s-1}}.$ By induction on $s$, there exists an edge $E' \in \mathcal{E}(\overline{\mathcal{H}}_{\mathcal{L}_{s-1}})$ such that $E' \subset E.$ Since $x_{u_s,c_s} \notin E$, we get $E' \in \mathcal{E}(\overline{\mathcal{H}}_{\mathcal{L}_s}).$ 
\end{proof}

In view of Lemma ~\ref{L-Constructible}, $\{x_{j,f}\}$ is constructible in $\overline{\mathcal{H}}_{\mathcal{L}_s}$ if and only if $\{x_{j,f}\} \in \mathcal{E}(\overline{\mathcal{H}}_{\mathcal{L}_s}).$

The following lemmas are useful in proving Theorem \ref{Th-MainTheorem}.

\begin{lem}\label{L-Edge1}\rm
    With notation as in Set-up \ref{Set-up} and Remark \ref{R-BasicRemark1}, let 
    $s \in[\alpha(\mathcal{L})-1].$ 
    Suppose that $|(F_{\ell_s}\cap G_{\ell_s})\setminus A_s| =1.$ 
    Then $\{x_{h_1,c_{s+1}},x_{h_2,e_2},\dots,x_{h_a,e_a}\}$ is an edge in $(\overline{\mathcal{H}}_{\mathcal{L}_s})^\circ.$
\end{lem}

\begin{proof}
	Let $E = \{x_{h_1,c_{s+1}},x_{h_2,e_2},\dots,x_{h_a,e_a}\}.$
	Then, by Remark \ref{R-BasicRemark2}(a), $E$ is constructible in $\overline{\mathcal{H}}_{\mathcal{L}_s}.$
	By Lemma \ref{L-Constructible}, there exists an edge 
	$E'$ in $\overline{\mathcal{H}}_{\mathcal{L}_s}$ such that 
	$E' \subset E.$ Further, since $\{x_{h_1,e_1},\dots,x_{h_a,e_a}\}$ is an edge in $\overline{\mathcal{H}}_{\mathcal{L}_s}$, it follows that $x_{h_1,c_{s+1}} \in E'.$ Without loss of generality, we assume that 
	$$E' =\{x_{h_1,c_{s+1}},x_{h_2,e_2},\dots,x_{h_b,e_b}\}.$$ 
    By Remark \ref{R-BasicRemark2}(b), we get $E' \in \mathcal{E}((\overline{\mathcal{H}}_{\mathcal{L}_s})^\circ)$ which gives $b >1.$
	Again by Lemma \ref{L-Constructible}, there exists $i \in [t]$ such that 
	$\{x_{h_1},\dots,x_{h_b}\} = F_{i}\setminus A_s$ and 
	$\mathbf{f} =(c_{s+1},e_2,\dots,e_b)$ is in $[k_i]^b$ with $|\mathbf{f}| \le k_i^{(s)}+b-1.$ If $i<\ell_s$, then $F_i \cap F_{\ell_s} \subset F_i\cap G_i.$ This implies that 
	$(F_i\cap G_i)\setminus A_{s-1}\ne\emptyset.$ Thus, we have $i \in U_s$, a contradiction. 
	On the other hand, if 
	$i >\ell_s$, then 
	$$(F_{\ell_s}\cap F_i)\setminus A_s =\{x_{h_1},\dots,x_{h_b}\}\subset (F_{\ell_s}\cap G_{\ell_s})\setminus A_s.$$ 
	Thus, we get $|(F_{\ell_s}\cap G_{\ell_s})\setminus A_s| > 1$, again a contradiction. Therefore, we obtain $i =\ell_s$, i.e. $E' =E.$
\end{proof}

The following example illustrate the fact that $\{x_{h_1,c_{s+1}},x_{h_2,e_2},\dots,x_{h_a,e_a}\}$ need not be an edge of hypergraph $(\overline{\mathcal{H}}_{\mathcal{L}_s})^\circ$ if $|(F_{\ell_s}\cap G_{\ell_s})\setminus A_s|>1.$\\

\begin{minipage}{\linewidth}
	\hspace{-0.5cm}	
	\begin{minipage}{0.59\linewidth}
		\begin{exmp}\label{E-Edge}
			Consider the simplicial tree $\Delta$ with facets $F_1 = \{x_1,x_2,x_3\}$, 
			$F_2 = \{x_4,x_5,x_6\}$, $F_3 = \{x_1,x_7,x_8\}$, $F_4 = \{x_7,x_8,x_9,x_{10}\}$ and $F_5=\{x_4,x_9,x_{10}\}$ as shown in Figure \ref{Fig 3}. 
			Let $\mathbf{k}_5=(1,1,1,4,2)\in \mathbb{N}^5_{>0}$ and $\mathcal{L} = P_0 P_1 P_2 \dots$ be an infinite string such that $P_1= P_2= P_3= L.$ Suppose $\mathcal{H} =\mathcal{H}(\Delta).$ With notation as in Set-up \ref{Set-up}, the sequence $(x_{u_1,c_1},x_{u_2,c_2},x_{u_3,c_3}) = (x_{1,1},x_{4,1},x_{7,1})$ is such that    
		\end{exmp}
	\end{minipage}
    \hspace{0.5cm}
    \begin{minipage}{0.26\linewidth}
    	\begin{figure}[H]
    		\begin{tikzpicture}[scale=0.6][line cap=round,line join=round,>=triangle 45,x=1.0cm,y=1.0cm]
    			\fill[lightgray] (0,0) -- (2,2) -- (2,0) -- cycle;
    			\fill[lightgray] (2,0) -- (3.86,2.12) -- (4,0) -- cycle;
    			\fill[lightgray] (4,0) -- (7,0) -- (3.86,2.12) -- (5.06,0.6) -- cycle;
    			\fill[lightgray] (3.86,2.12) -- (4,0) -- (5.06,0.6) -- cycle;
    			\fill[lightgray] (5.06,0.6) -- (7,0) -- (4,0) -- cycle;
    			\fill[lightgray] (7,0) -- (7,2) -- (5.06,0.6) -- cycle;
    			\fill[lightgray] (8,0) -- (7,2) -- (9,2) -- cycle;
    			\draw (0,0)-- (2,2);
    			\draw (2,2)-- (2,0);
    			\draw (2,0)-- (0,0);
    			\draw (2,0)-- (3.86,2.12);
    			\draw (3.86,2.12)-- (4,0);
    			\draw (4,0)-- (2,0);
    			\draw (4,0)-- (7,0);
    			\draw (7,0)-- (3.86,2.12);
    			\draw (3.86,2.12)-- (5.06,0.6);
    			\draw (5.06,0.6)-- (4,0);
    			\draw (3.86,2.12)-- (4,0);
    			\draw (4,0)-- (5.06,0.6);
    			\draw (5.06,0.6)-- (3.86,2.12);
    			\draw (5.06,0.6)-- (7,0);
    			\draw (7,0)-- (4,0);
    			\draw (4,0)-- (5.06,0.6);
    			\draw (7,0)-- (7,2);
    			\draw (7,2)-- (5.06,0.6);
    			\draw (5.06,0.6)-- (7,0);
    			\draw (8,0)-- (7,2);
    			\draw (7,2)-- (9,2);
    			\draw (9,2)-- (8,0);
    			
    			\begin{scriptsize}
    				\fill  (0,0) circle (1.5pt);
    				\draw[below] (0,0) node {$x_3$};
    				\fill  (2,0) circle (1.5pt);
    				\draw[below] (2,0) node {$x_1$};
    				\fill  (2,2) circle (1.5pt);
    				\draw[above] (2,2) node {$x_2$};
    				\fill  (3.86,2.12) circle (1.5pt);
    				\draw[above] (3.86,2.12) node {$x_7$};
    				\fill  (4,0) circle (1.5pt);
    				\draw[below] (4,0) node {$x_8$};
    				\fill  (7,0) circle (1.5pt);
    				\draw[below] (7,0) node {$x_{10}$};
    				\fill  (5.06,0.6) circle (1.5pt);
    				\draw[left] (5.06,0.6) node {$x_9$};
    				\fill  (7,2) circle (1.5pt);
    				\draw[above] (7,2) node {$x_4$};
    				\fill  (9,2) circle (1.5pt);
    				\draw[above] (9,2) node {$x_5$};
    				\fill (8,0) circle (1.5pt);
    				\draw[below] (8,0) node {$x_6$};
    			\end{scriptsize}
    		\end{tikzpicture}
    		\caption{}
    		\label{Fig 3}
    	\end{figure}
    \end{minipage}
\end{minipage}
$$
\begin{array}{ccl}
	\overline{\mathcal{H}}_{\mathcal{L}_q}& = &
	\left\{\begin{array}{lll}
		\overline{\mathcal{H}}_{\mathcal{L}_{q-1}}/{x_{u_q,c_q} }&~\mbox{if}~P_q = L;\\
		\overline{\mathcal{H}}_{\mathcal{L}_{q-1}} \setminus {x_{u_q,c_q} } &~\mbox{otherwise}
	\end{array}\right.
\end{array}
$$
for all $q \in [3].$
The edge set of the hypergraph $(\overline{\mathcal{H}}_{\mathcal{L}_3})^\circ$ is given by\\
    
    $\mathcal{E}((\overline{\mathcal{H}}_{\mathcal{L}_3})^\circ)=
    \left\{\begin{array}{ll}
    	&\{x_{2,1},x_{3,1}\},\{x_{5,1},x_{6,1}\},\{x_{8,2},x_{9,3},x_{10,1}\},\{x_{8,2},x_{9,1},x_{10,3}\}, \nonumber \\ 
    	&
    	\{x_{8,2},x_{9,2},x_{10,2}\},\{x_{9,1},x_{10,1}\},\{x_{9,2},x_{10,1}\},\{x_{9,1},x_{10,2}\}
    \end{array}\right\}.$\\ 
    Now, $U_3 =\{4\}$, and hence $\ell_3 = 4.$ Since 
    $G_4 = F_5$, we get 
    $$(F_4\cap G_4)\setminus \{x_1,x_4,x_7\} =\{x_9,x_{10}\}.$$
    Thus, we can take $x_{u_4,c_4}=x_{9,1}.$ Note that $\{x_{8,2},x_{9,3},x_{10,1}\}$ is an edge  $(\overline{\mathcal{H}}_{\mathcal{L}_3})^\circ$ but $\{x_{8,2},x_{9,1},x_{10,1}\}$ is not an edge in $(\overline{\mathcal{H}}_{\mathcal{L}_3})^\circ.$

\begin{lem}\label{L-Edge2}\rm
	With notation as in Set-up \ref{Set-up} and Remark \ref{R-BasicRemark1}, let $s \in[\alpha(\mathcal{L})-1].$ Suppose that 
	$|(F_{\ell_s}\cap G_{\ell_s})\setminus A_s|>1.$ Then
	$\{x_{h_1,c_s},x_{h_2,f_2},\dots,x_{h_a,f_a}\} \in \mathcal{E}((\overline{\mathcal{H}}_{\mathcal{L}_s})^\circ)$ for some 
	$(f_2,\dots,f_a) \in [k_{\ell_s}]^{a-1}.$
\end{lem}

\begin{proof}
    We have $(F_{\ell_s}\cap G_{\ell_s})\setminus A_s \subset \{x_{w_{j_0}},\dots,x_{w_r}\}.$ Thus, there exists some $j>j_0$ such that $x_{w_j}\in (F_{\ell_s}\cap G_{\ell_s})\setminus A_s.$ Choose 
    $j' = \min\{j:x_{w_j} \in (F_{\ell_s}\cap G_{\ell_s})\setminus A_s, j \neq j_0 \}.$ 
    Then $x_{w_{j'}} = x_{h_q}$ 
    for some $2 \le q\le a.$ Take $f_1 = c_{s+1}$, $f_q = e_1+e_q-c_{s+1}$ and $f_p = e_p$ for $p \ne 1, q.$ We have 
    $$|\mathbf{f}|= |\mathbf{e}|\le k_{\ell_s}^{(s)}+a-1, \text{ where } 
    \mathbf{f} =(f_1,\dots,f_a).$$ 
    Let $E = \{x_{h_1,f_1},\dots,x_{h_a,f_a}\}.$ Using Remark \ref{R-BasicRemark2}(a) and the fact that $f_q \ge e_q$, we get $E$ is constructible in $\overline{\mathcal{H}}_{\mathcal{L}_s}.$ We claim that $E \in \mathcal{E}((\overline{\mathcal{H}}_{\mathcal{L}_s})^\circ).$ If possible, let $E \notin \mathcal{E}((\overline{\mathcal{H}}_{\mathcal{L}_s})^\circ).$ Then by Lemma \ref{L-Constructible}, there exists an edge 
    $E' =\{x_{h_{p_1},f_{p_1}},\dots,x_{h_{p_b},f_{p_b}}\}$ in $\overline{\mathcal{H}}_{\mathcal{L}_s}$ such that $E'\subsetneq E$, where 
    $\{x_{h_{p_1}},\dots,x_{h_{p_b}}\} = F_i \setminus A_s$ and 
    $\mathbf{f}' =(f_{p_1},\dots,f_{p_b}) \in [k_i]^b$ with 
    $|\mathbf{f}'| \le k_i^{(s)}+b-1$ 
    for some $i \in [t].$
	We claim that $b>1.$ If possible, suppose $b=1.$ Then 
	$E' =\{x_{h_{p_1},f_{p_1}}\}$ and $f_{p_1}\le k_i^{(s)}.$ By Remark \ref{R-BasicRemark2}(b), 
	$x_{h_1,f_1}= x_{u_{s+1},c_{s+1}} \notin E'.$ 
	Since $\{x_{h_1,e_1},\dots,x_{h_a,e_a}\}\in \mathcal{E}((\overline{\mathcal{H}}_{\mathcal{L}_s})^\circ)$, it follows that 
	$E' =\{x_{h_q,f_q}\}.$
    Now, by using the inequalities $e_q \le f_q = f_{p_1} \le k_i^{(s)}$,we get $\{x_{h_q,e_q}\}$ is constructible in $\overline{\mathcal{H}}_{\mathcal{L}_s}.$ Thus, by using Lemma 
    \ref{L-Constructible}, we get 
    $\{x_{h_q,e_q}\}\in \mathcal{E}(\overline{\mathcal{H}}_{\mathcal{L}_s}).$ This contradicts the fact that $\{x_{h_1,e_1},\dots,x_{h_a,e_a}\}\in \mathcal{E}((\overline{\mathcal{H}}_{\mathcal{L}_s})^\circ).$ Thus, we must have $b>1.$ Next, we prove that $i>\ell_s.$ If $i<\ell_s$, then 
    $(F_i \cap F_{\ell_s})\setminus A_s \subset 
    (F_i\cap G_i)\setminus A_s$ which implies that $i\in U_s$, a contradiction. This proves that $i>\ell_s.$ Note that
    $$(F_{\ell_s}\cap F_i)\setminus A_s=\{x_{h_{p_1}},\dots,x_{h_{p_b}}\}\subset (F_{\ell_s}\cap G_{\ell_s})\setminus A_s.$$ 
    Since $b>1$, it follows from the definition of sequence of good vertices that $x_{w_{j_0}},x_{w_{j'}}\in F_i\setminus A_s$, i.e. $x_{h_1},x_{h_q}\in F_i\setminus A_s.$ Therefore, we have 
    $$|\mathbf{e}'| = |\mathbf{f}'| \le k_i^{(s)}+b-1,$$
    where $\mathbf{e}' =(e_{p_1},\dots,e_{p_b}).$
    Thus, $\{x_{h_{p_1},e_{p_1}},\dots,x_{h_{p_b},e_{p_b}}\}$ is constructible $\overline{\mathcal{H}}_{\mathcal{L}_s}.$ Now, by Lemma \ref{L-Constructible}, there exists an edge $E'' \in \mathcal{E}(\overline{\mathcal{H}}_{\mathcal{L}_s})$ such that $E''\subset \{x_{h_{p_1},e_{p_1}},\dots,x_{h_{p_b},e_{p_b}}\}.$ This implies that $\{x_{h_1,e_1},\dots,x_{h_a,e_a}\}\notin \mathcal{E}((\overline{\mathcal{H}}_{\mathcal{L}_s})^\circ)$, a contradiction.
\end{proof}

\begin{lem}\label{L-SheddingVertex}
	With notation  as in Set-up \ref{Set-up} and Remark \ref{R-BasicRemark1}, let $s \in [\alpha(\mathcal{L})-1]\cup\{0\}.$
    Then $x_{u_{s+1},c_{s+1}}$ is a shedding vertex of hypergraph  $(\overline{\mathcal{H}}_{\mathcal{L}_s})^\circ.$ 
\end{lem}

\begin{proof}
    Let $\mathcal{D}=\Delta((\overline{\mathcal{H}}_{\mathcal{L}_s})^\circ)$ and $W\in\link_{\mathcal{D}}(x_{u_{s+1},c_{s+1}}).$ In view of Lemmas \ref{L-Edge1} and \ref{L-Edge2}, we may assume that $e_1 = c_{s+1}.$
	For $p\in [a]$, let $\sigma_p=\max\{c:x_{h_p,c}\in \mathcal{V}((\overline{\mathcal{H}}_{\mathcal{L}_s})^\circ)\}.$ Suppose that $$\{x_{h_p,1},\dots,x_{h_p,\sigma_p}\}\cap \mathcal{V}((\overline{\mathcal{H}}_{\mathcal{L}_s})^\circ)\subset W$$ 
	for all $2\le p\le a.$ Then $\{x_{h_1,e_1},x_{h_2,e_2},\dots,x_{h_a,e_a}\}\subset W\cup\{x_{u_{s+1},c_{s+1}}\}$, a contradiction. Thus, there exists $2\le p\le a$ such that  $\{x_{h_p,1},\dots,x_{h_p,\sigma_p}\}\cap \mathcal{V}((\overline{\mathcal{H}}_{\mathcal{L}_s})^\circ)\not{\subset} W.$ Suppose $f$ is the largest integer such that $x_{h_p,f}\in \mathcal{V}((\overline{\mathcal{H}}_{\mathcal{L}_s})^\circ)\setminus W.$ We claim that $W \cup\{x_{h_p,f}\} \in \del_{\mathcal{D}}(x_{u_{s+1},c_{s+1}}).$ On the contrary, suppose that 
	$W \cup\{x_{h_p,f}\} \notin \del_{\mathcal{D}}(x_{u_{s+1},c_{s+1}}).$ Since $x_{u_{s+1},c_{s+1}}\notin W$, it follows that there exists an edge $E = \{x_{j_1,f_1},\dots,x_{j_b,f_b}\}$ in $(\overline{\mathcal{H}}_{\mathcal{L}_s})^\circ$ such that $E \subset W\cup \{x_{h_p,f}\}.$ In view of Lemma \ref{L-Constructible}, there exists $i\in [t]$ such that $\{x_{j_1},\dots,x_{j_b}\} = F_i\setminus A_s$ and $\mathbf{f} =(f_1,\dots,f_b) \in [k_i]^b$ with 
	$$f_q > c_m \text{ if } x_{j_q} = x_{u_m} \in B_s\setminus A_s \text{ and } |\mathbf{f}| \le k_i^{(s)}+b-1.$$
    Since $W\in \mathcal{D}$, we have $x_{h_p,f} \in E.$ Let $x_{h_p,f} =x_{j_1,f_1}.$ Then $x_{h_p} =x_{j_1}$ and $f =f_1.$ If $i<\ell_s$, then the fact $x_{h_p}\in (F_i\cap F_{\ell_s})\setminus A_s$, implies that $i\in U_s.$ This contradicts the minimality of $\ell_s.$ Thus, we get 
    $i\ge \ell_s$, and hence 
    $x_{h_p} \in (F_{\ell_s}\cap G_{\ell_s}) \setminus A_s.$ 
    Recall that 
    $$j_0 = \min\{j:x_{w_j} \in (F_{\ell_s}\cap G_{\ell_s}) \setminus A_s\}.$$ 
    Thus, $x_{h_p} = x_{w_j}$ for some $j > j_0.$ Further, it follows from the definition of sequence of good vertices that $x_{u_{s+1}}= x_{w_{j_0}}\in F_i.$ Let $x_{u_{s+1}} = x_{j_2}.$ Then $f_2>c_{s+1} \ge 1.$ Now, the inequalities $|\mathbf{f}| \le k_i^{(s)}+b-1 \le k_i+b-1$ implies that $f_1 <k_i.$ Thus, the set 
    $$F = \{x_{j_1,f_1+1} = x_{h_p,f+1},x_{j_2,c_{s+1}} = x_{u_{s+1},c_{s+1}},\dots,x_{j_b,f_b}\}\subset \mathcal{V}(\mathcal{H}(\bk_t)).$$ Note that 
    $$(f_1+1)+c_{s+1}+ \cdots +f_b \le|\mathbf{f}| \le k_i^{(s)}+b-1.$$ 
    By Remark \ref{R-BasicRemark2}(a), $F$ is constructible in $\overline{\mathcal{H}}_{\mathcal{L}_s}$, and hence by
    Lemma \ref{L-Constructible}, there exists an edge $F' \in \mathcal{E}(\overline{\mathcal{H}}_{\mathcal{L}_s})$ such that $F'\subset F.$  Since $E$ is an edge in $(\overline{\mathcal{H}}_{\mathcal{L}_s})^\circ$, we have
    $$
    	\{x_{j_1,f_1+1},x_{j_2,c_{s+1}}\}\cap F' \ne \emptyset.
    $$
    We prove that $F'$ is an edge in $(\overline{\mathcal{H}}_{\mathcal{L}_s})^\circ.$
    If $F'\notin \mathcal{E}((\overline{\mathcal{H}}_{\mathcal{L}_s})^\circ)$, then by Remark \ref{R-BasicRemark2}(b), it follows that $x_{j_2,c_{s+1}}\notin F'.$ Therefore, we have 
    $F' =\{x_{j_1,f_1+1}\}.$ Now, Lemma \ref{L-Constructible} implies that 
    $\{x_{j_1,f_1}\} \in \mathcal{E}(\overline{\mathcal{H}}_{\mathcal{L}_s}).$ 
   Hence $E\notin\mathcal{E}((\overline{\mathcal{H}}_{\mathcal{L}_s})^\circ)$,
   a contradiction. This proves that $F'$ is an edge in $(\overline{\mathcal{H}}_{\mathcal{L}_s})^\circ.$ Now, we prove that 
   $F'\subset W \cup\{x_{u_{s+1},c_{s+1}}\}.$ 
   If $x_{j_1,f_1+1}\notin F'$, then we are done.
   On the other hand, if $x_{j_1,f_1+1}\in F'$, then by maximality of $f_1 =f$, we get $x_{j_1,f_1+1}\in W$, and hence $F'\subset W \cup\{x_{u_{s+1},c_{s+1}}\}.$ Thus, in both cases, we have $F'\subset W \cup\{x_{u_{s+1},c_{s+1}}\}$, and hence 
   $W \cup\{x_{u_{s+1},c_{s+1}}\} \notin \mathcal{D}$, a contradiction. This proves our claim that $W \cup\{x_{h_p,f}\} \in \del_{\mathcal{D}}(x_{u_{s+1},c_{s+1}}).$ Thus, $x_{u_{s+1},c_{s+1}}$ is a shedding vertex of $(\overline{\mathcal{H}}_{\mathcal{L}_s})^\circ.$
\end{proof}

\begin{lem}\label{L-Alpha>}
	With notation as in Set-up \ref{Set-up} and Remark \ref{R-BasicRemark1},
	$\overline{\mathcal{H}}_{\mathcal{L}_s}$ is a vertex decomposable hypergraph for all $s\ge \alpha(\mathcal{L}).$
\end{lem}

\begin{proof}
    For simplicity, we write $\alpha =\alpha(\mathcal{L}).$ Then $U_{\alpha} =\emptyset.$ It is suffices to prove that $\overline{\mathcal{H}}_{\mathcal{L}_{\alpha}}$ is a vertex decomposable hypergraph. If $\overline{\mathcal{H}}_{\mathcal{L}_{\alpha}}$ is an isolated hypergraph, then it is vertex decomposable. Now, let $\overline{\mathcal{H}}_{\mathcal{L}_{\alpha}}$ is not an isolated hypergraph. Then, by Lemma \ref{L-Constructible}, there exists $i \in[t]$ such that the following condition holds:
    \begin{enumerate}[($\star$)]
    	\item if $F_i\setminus{A_{\alpha}}=\{x_{j_1},\dots,x_{j_b}\}$, then  $\{x_{j_1,f_1},\dots,x_{j_b,f_b}\}$ is an edge in $(\overline{\mathcal{H}}_{\mathcal{L}_{\alpha}})^\circ$ for some 
    	$\mathbf{f} =(f_1,\dots,f_b)$ in $[k_i]^b$ with $f_p > c_m$ if $x_{j_p} = x_{u_m} \in B_{\alpha}\setminus A_{\alpha}$ and $|\mathbf{f}| \le k_i^{(\alpha)}+b-1.$ 
    \end{enumerate}
    Let $A$ be the set of all $i \in [t]$ satisfying condition $(\star).$ If possible, let there exist $i,i'\in A$ with $i' >i$ such that $(F_i\setminus A_{\alpha})\cap(F_{i'}\setminus A_{\alpha})\neq\emptyset.$ Let $x \in(F_i\setminus A_{\alpha})\cap(F_{i'}\setminus A_{\alpha}).$ Then 
    $x \in (F_i\cap G_i)\setminus A_{\alpha}.$ This implies that $i \in U_{\alpha}$, a contradiction. Hence 
    \begin{equation}\label{EQ2}
        (F_i\setminus A_{\alpha})\cap(F_{i'}\setminus A_{\alpha}) =\emptyset 
        \text{ for all } i,i'\in A \text{ with } i'\neq i.
    \end{equation}
    Let $i \in A$ and $F_i\setminus A_{\alpha} =\{x_{j_1}\dots x_{j_b}\}.$ For $p \in [b]$, set 
    $$c'_p =\min\{c:x_{j_p,c} \in \mathcal{V}((\overline{\mathcal{H}}_{\mathcal{L}_{\alpha}})^\circ)\}-1.$$ 
    Now, let
    $\bar{d}_i =\sum_{p \in [b]} c'_p$ and $\bar{k}_i =\max\{0,k^{(\alpha)}_i-\bar{d}_i\}.$ 
    We claim that 
    $(\overline{\mathcal{H}}_{\mathcal{L}_{\alpha}})^\circ \simeq \mathcal{H}'$, where $$\mathcal{H}'=\bigsqcup_{i \in A}(F_i\setminus A_{\alpha})(\bar{k}_i).$$
    Let $\mathbf{f} =(f_1,\dots,f_b) \in [\bar{k}_i]^b.$ Then 
    $$|\mathbf{f}| \le \bar{k}_i+b-1 \text{ if and only if } 
    |\mathbf{f}'| \le k^{(\alpha)}_i+b-1,$$ where $\mathbf{f}'=(f_1+c'_1,\dots,f_b+c'_b)\in [k_i]^b.$ Now, let 
    $E =\{x_{j_1,f_1},\dots,x_{j_b,f_b}\}$ is an edge in $\mathcal{H}'.$ If $x_{j_p}=x_{u_m}\in B_{\alpha}\setminus A_{\alpha}$, then $f'_p\ge c'_p+1>c_m$ since $x_{j_p,c'_p+1}\in \mathcal{V}((\overline{\mathcal{H}}_{\mathcal{L}_{\alpha}})^\circ).$ 
   Thus, we observe that $E' =\{x_{j_1,f'_1},\dots,x_{j_b,f'_b}\}$ is constructible in $\overline{\mathcal{H}}_{\mathcal{L}_{\alpha}}.$ By Lemma \ref{L-Constructible}, there exists an edge $E''$ in $\overline{\mathcal{H}}_{\mathcal{L}_{\alpha}}$ such that $E'' \subset E'.$ Without loss of generality, we may assume that 
   $E'' =\{x_{j_1,f'_1},\dots,x_{j_{b'},f'_{b'}}\}$ for some $b'\le b.$ Again by Lemma \ref{L-Constructible}, there exists $i'\in [t]$ such that $\{x_{j_1},\dots,x_{j_{b'}}\} =F_{i'}\setminus A_{\alpha}$ and 
   $$\mathbf{\bar{f}} =(f'_1,\dots,f'_{b'})\in [k_{i'}]^{b'} \text{ with } |\mathbf{\bar{f}}| \le k^{(\alpha)}_{i'}+b'-1.$$
   If possible, let $b'=1.$ Then $E' =\{x_{j_1,f'_1}\}.$ Since $f'_1 \ge c'_1+1$, it follows from Lemma \ref{L-Constructible} that $\{x_{j_1,c'_1+1}\} \in \mathcal{E}(\overline{\mathcal{H}}_{\mathcal{L}_{\alpha}})$, and hence $x_{j_1,c'_1+1} \notin \mathcal{V}((\overline{\mathcal{H}}_{\mathcal{L}_{\alpha}})^\circ)$, a contradiction. Thus, we have $b'>1$, i.e. $i' \in A.$
   Now, Equation \eqref{EQ2} implies that $i =i'$, i.e. $E' =E''.$ Therefore, $E'$ is an edge in 
   $(\overline{\mathcal{H}}_{\mathcal{L}_{\alpha}})^\circ.$
    Conversely, if $E' =\{x_{j_1,f'_1},\dots,x_{j_b,f'_b}\} \in \mathcal{E}((\overline{\mathcal{H}}_{\mathcal{L}_{\alpha}})^\circ)$, then $\{x_{j_1,f_1},\dots,x_{j_b,f_b}\} \in \mathcal{E}(\mathcal{H}').$ This proves our claim. Now, using Lemma~\ref{L-Simplex}, we can say that hypergraph $\overline{\mathcal{H}}_{\mathcal{L}_{\alpha}}$ is vertex decomposable.
\end{proof}

\begin{exmp}\label{E6}
	Let $\Delta$ be a simplicial tree of Figure \ref{Fig 2}. In Example \ref{E-illustrative}, we have seen that $\overline{\mathcal{H}}_{\mathcal{L}_4}$ is a hypergraph with the edge set
	
	$$\mathcal{E}((\overline{\mathcal{H}}_{\mathcal{L}_4})^\circ)=
	\Bigl\{\{x_{2,1},x_{3,1}\},\{x_{7,1},x_{8,1}\}\Bigr\}.$$

    We observe that $\alpha(\mathcal{L}) =4$ and 
	$(\overline{\mathcal{H}}_{\mathcal{L}_4})^\circ \simeq (F_1\setminus A_4)(1)\sqcup (F_3\setminus A_4)(1).$
\end{exmp}

\begin{lem}\label{L-Alpha}
	With notation as in Set-up \ref{Set-up} and Remark \ref{R-BasicRemark1},
	the hypergraph $\overline{\mathcal{H}}_{\mathcal{L}_s}$ is vertex decomposable for all $s\in \mathbb{N}$ and for all infinite string $\mathcal{L}.$
\end{lem}

\begin{proof}
	Let $\alpha_0=\max\{\alpha(\mathcal{L}):\mathcal{L}~\mbox{is an infinite string}\}.$ In view of Lemma \ref{L-Alpha>}, it is enough to prove that $\overline{\mathcal{H}}_{\mathcal{L}_s}$ is a vertex decomposable hypergraph  for all 
	$s \le \alpha_0$ and for all infinite string $\mathcal{L}.$ We proceed by using induction on $\alpha_0-s.$ The result holds for $s =\alpha_0$ by Lemma \ref{L-Alpha>}. Now, we suppose that $s <\alpha_0$ and
	$\mathcal{L} = P_0P_1 P_2\dots P_s P_{s+1}\dots$ be any infinite string. If $\overline{\mathcal{H}}_{\mathcal{L}_s}=
	\overline{\mathcal{H}}_{\mathcal{L}_{s+1}}$, then the result follows from induction. On the other hand, suppose that $\overline{\mathcal{H}}_{\mathcal{L}_s}\ne
	\overline{\mathcal{H}}_{\mathcal{L}_{s+1}}$
	and $P_{s+1} = D.$ Then 
	$s \in [\alpha(\mathcal{L})-1]$ and 
	$$\left((\overline{\mathcal{H}}_{\mathcal{L}_s})^\circ \setminus {x_{u_{s+1},c_{s+1}}}\right)^\circ = (\overline{\mathcal{H}}_{\mathcal{L}_{s+1}})^\circ.$$ 
	Now, consider the string 
	$\mathcal{L}' = P_0P_1 P_2\dots P_s P'_{s+1}\dots$ with $P'_{s+1} = L.$ We have $\mathcal{L}'_s=\mathcal{L}_s$, and hence $$\left((\overline{\mathcal{H}}_{\mathcal{L}_s})^\circ / x_{u_{s+1},c_{s+1}}
	\right) ^\circ  = (\overline{\mathcal{H}}_{\mathcal{L}'_{s+1}})^\circ.$$
	Thus by induction, $(\overline{\mathcal{H}}_{\mathcal{L}_s})^\circ \setminus{x_{u_{s+1},c_{s+1}}}$
	and $(\overline{\mathcal{H}}_{\mathcal{L}_s})^\circ /{x_{u_{s+1},c_{s+1}}}$
	are vertex decomposable hypergraphs. 
	By Lemma \ref{L-SheddingVertex}, $\overline{\mathcal{H}}_{\mathcal{L}_s}$ is vertex decomposable hypergraph. In similar manner, we can prove the result when $P_{s+1} =L.$
\end{proof}

Now we prove the first main result of the paper.
\begin{proof}[Proof of Theorem \ref{Th-MainTheorem}]
	In view of Lemma \ref{L-Alpha}, the hypergraphs $\mathcal{H}(\mathbf{k}_t)/{x_{u_1,1}}$ and $\mathcal{H}(\mathbf{k}_t) \setminus {x_{u_1,1}}$ are vertex decomposable. By Lemma \ref{L-SheddingVertex}, $x_{u_1,1}$ is a shedding vertex of hypergraph $\mathcal{H}(\mathbf{k}_t).$ Hence $\mathcal{H}(\mathbf{k}_t)$ is a vertex decomposable hypergraph.
\end{proof}

%

\begin{proof}[Proof of Theorem \ref{Th-1stMainResult}]
	Let $\mathcal{H} =\mathcal{H}(\Delta).$ Using \cite[Corollary 1.6]{Herzog2006SGV}, we have $J(\Delta)^{k} = J(\Delta)^{(k)}$, and hence $J(\Delta)^{k} =J(\mathcal{H})^{(k)}.$
	In view of Theorem \ref{Th-MainTheorem}, the hypergraph $\mathcal{H}(\mathbf{k})$ is vertex decomposable, and hence shellable. Now by Lemma \ref{L-Polarisation}, we have $\widetilde{J(\Delta)^{k}} = J(\mathcal{H}(\mathbf{k})).$ Thus by \cite[Theorem 8.2.5]{HBook}, $\widetilde{J(\mathcal{H})^{k}}$ has linear quotients. By \cite[Lemma 3.5]{Fakhari}, $J(\Delta)^{k}$ has linear quotients. Hence by \cite[Theorem 8.2.15]{HBook}, $J(\Delta)^{k}$ is componentwise linear.  
\end{proof}

Let $\Delta$ be a simplicial tree and  $\deg(J(\Delta))$ be the maximum degree of minimal monomial generators of $J(\Delta).$ Then using Theorem \ref{Th-1stMainResult}, we obtain the following result.

\begin{thm}\label{Th-Reg}
	Let $\Delta$ be a simplicial tree. Then 
	$\reg(J(\Delta)^s) = s \deg(J(\Delta))$ for $s \ge 1.$
\end{thm}
We are now in position to prove the second main result of this paper.		
\begin{proof}[Proof of Theorem \ref{Th-2ndMainResult}]
	Equivalence of (i), (ii) and (iii) follows from Theorem \ref{Th-Reg}, and equivalence of (i) and (iv) follows from \cite[Thereom 3]{eagon}. Further, equivalence of (iv), (v) and (vi) follows from \cite[Corollary 7.8]{Faridi05} and \cite[Corollary 8.3]{Faridi05}.
\end{proof}

\noindent
{\bf Acknowledgements.}
The second author acknowledges the financial support from DST, Govt of India under the Start-Up Research Grant (SRG/2021/001442). The second author was also supported by the seed grant received from Indian Institute of Technology, Jammu.

\renewcommand{\bibname}{References}

\renewcommand{\bibname}{References}
\bibliographystyle{plain}  

\bibliography{refs_reg}

\end{document}